\documentclass{article} %

\usepackage[utf8]{inputenc}
\usepackage{amssymb}
\usepackage{comment}
\usepackage{mathtools}
\usepackage{enumerate}
\usepackage{esint}
\usepackage{graphicx}
\usepackage{hyperref}
\usepackage{numprint}
\usepackage{stmaryrd}
\usepackage{stackrel}
\usepackage{scalerel}
\usepackage{soul}
\usepackage[labelformat=simple]{subcaption}

\usepackage{wrapfig}
\newcommand{\bLam}{\boldsymbol{\Lambda}}

\usepackage{amsthm}
\newtheorem{theorem}{Theorem}
\newtheorem{remark}{Remark}
\newtheorem{corollary}{Corollary}
\newtheorem{lemma}{Lemma}

\usepackage{booktabs} 

\usepackage{pgfplots}
\pgfplotsset{compat=1.16}
\usepackage{pgfplotstable}
\usepackage{tikz}
\usepackage{tkz-euclide}
\usepgfplotslibrary{external,statistics}
\usepgfplotslibrary{colorbrewer}

\usepackage[colorinlistoftodos]{todonotes}
\makeatletter
\renewcommand{\todo}[2][]{\tikzexternaldisable\@todo[#1]{#2}\tikzexternalenable}
\makeatother

\usepackage{custom-commands-private}

\begin{document}

\title{Mixed-dimensional modeling of vascular tissues with reduced Lagrange multipliers}
\author{
	Camilla Belponer\thanks{Institute of Mathematics, Universität Augsburg, Augsburg, Germany; Weierstrass Institute for Applied Mathematics and Stochastics (WIAS), Berlin, Germany}
	\and
	Alfonso Caiazzo\thanks{Weierstrass Institute for Applied Mathematics and Stochastics (WIAS), Berlin, Germany}
	\and
	Luca Heltai\thanks{Department of Mathematics, University of Pisa, Pisa, Italy}
}
\date{}

\maketitle

\section*{Abstract}
This paper presents a numerical method for the simulation of multiscale materials composed of an elastic matrix and slender active inclusions. 
The setting is motivated by the modeling of vascularized tissues and by problems arising in the context of medical imaging techniques, 
where the estimation of effective (i.e., macroscale) material properties is affected by the presence of microscale structures and
microscale dynamics, such as fluid flow in the vasculature. 
We propose a method where the background solid material and the active slender inclusions are discretized independently, imposing the required interface conditions via non-matching Lagrange multipliers. 
The intrinsic geometrical complexity of the resulting computational model is simplified by relying on a reduced Lagrange multiplier framework, where the functional space of the Lagrange multiplier is replaced by the tensor product between an infinite dimensional Sobolev space defined on a lower-dimensional characteristic set of co-dimension two, and a finite dimensional space defined on the cross-sections of the inclusions.
In view of the coupling with one-dimensional blood flow models, we derive a non-standard boundary condition that enforces a local deformation on the solid-fluid boundary, and we present the details of its stability analysis in the continuous elasticity setting.
The method is validated with different numerical examples in two and three dimensions, assessing its convergence properties and its potential
for the in silico characterization of tissues samples.

\paragraph{keywords}{Finite element method, linear elasticity, multiscale methods, model reduction, immersed interfaces, Lagrange multipliers}


\section{Introduction}

This paper focuses on the computational multiscale modeling of elastic materials. The topic is motivated by applications in the context of biological tissue imaging, such as multiparametric MRI with diffusion-weighted imaging \cite{jajamovich-etal-2014,lewin-etal-2007}, or magnetic resonance elastography (MRE) \cite{muthupillai_1995,sack-bioqic-18,sack-2023-review}, where tissue data and images are combined with physical and computational physical models to estimate mechanical and biophysical properties.
These \textit{inverse} problems are extremely challenging for multiple reasons. 
On the one hand, the underlying forward problems require the handling of the interaction between the solid matrix and the fluid vasculature. Besides the complexity of handling numerically
the fluid-solid coupling at the interface, fully resolved models are extremely prohibitive due to the geometrical complexity at the small scales of vascular structures.
On the other hand, especially in the context of medical images, data are typically available only at relatively coarse scales (the typical resolution of MRI is of the order of millimeters).
To bridge the gap between the required model complexity and
available data resolution, tissue models based on linearized elasticity and poroelasticity are commonly used in the context of 
medical imaging to characterize mechanical and constitutive \textit{effective} parameters.

However, modeling the smaller scales cannot be neglected if the characterization of fluid phase is of interest for the final application. 
An example is the possibility of characterizing  the effect on macroscopic-biophysical parameters of the variations of the fluid pressure within the tissue\cite{lilaj_etal_2021a,JAITNER2025312}.
Simplified coarse-scale models assuming homogeneous tissue properties might 
lead to inaccurate estimations. A concrete example on the sensitivity of liver tissue parameters to intrinsic poroelastic properties and vascular architecture has been
recently presented and discussed in the experimental study presented by Safraou et al. \cite{safraou-etal-2023}, investigating
the influence of static portal pressure on liver stiffness (see also \cite{millonig-etal-2010,palaniyappan-etal-2016}).
It is thus necessary to take into account
multiscale descriptions, in which \textit{effective} macroscale tissue dynamics intrinsically account for the small scales via suitable surrogate models, i.e.,  capable of retaining the details of the microscale vasculature.

Our approach is based on \textit{mixed-dimensional} partial differential equations, namely, modeling  a tissue sample is modeled as a linear, elastic matrix coupled to an arbitrary vascular fluid structures of co-dimension two (i.e., 3D-1D or 2D-0D).
The assumption behind the usage of mixed-dimensional methods is that 
part of the physics can be sufficiently well approximated on a lower-dimensional manifold, e.g., using one-dimensional blood flow models.
This type of models has been first analyzed in \cite{dangelo-quarteroni-2008} in the context of diffusion, and later on applied also to perfusion, porous media flows
\cite{cattaneo-2014a,cattaneo-2014b,koch-2022}, as well as 
elasticity\cite{heltai-caiazzo-2019,heltai-caiazzo-mueller-2021}.

The model considered in this paper is inspired by the one presented
by Heltai et al.\cite{heltai-caiazzo-2019,heltai-caiazzo-mueller-2021}, in which the coupling was implemented via a singular forcing term in the elasticity equation, but restricted to Neumann-like boundary condition for the solid model,
computed from an asymptotic approximation of a local analytical solution.
The goal of this work is to proposed, analyze, and validate a 
mixed-dimensional model for general coupling conditions between tissue and fluid, as a further step towards the treatment of fully coupled 3D-1D fluid-structure interaction problems.

We handle the coupling using a non-matching immersed method based on
a Lagrange multipliers (LM) approach\cite{babuska-1973,bramble-1981}. 
In this method the variational formulation is defined based on a minimization problem equivalent to the original PDE, in which the boundary or coupling conditions are imposed weakly as a set of constraints, using functional spaces defined on the interfaces.
LM formulations play an important role in the numerical solution of partial differential equations using the finite element method,
in particular in the context of coupled multiphysics and multiscale models (see, e.g., \cite{boffi-2021,boffi-2022,franceschini-2021}
for some recent examples). 
Recent LM formulations for mixed-dimensional material models were  proposed, e.g., in \cite{AlzettaHeltai-2020-a},
considering the coupling of a three-dimensional bulk mechanical problem with one-dimensional fiber structures. 
Preliminary stability results for Dirichlet\--Neumann coupling conditions on mixed--dimensional (3D\--1D) problems using LM
were recently presented by Kuchta et al.\cite{Kuchta2020}, while following related works discussed suitable preconditiong strategies\cite{Kuchta2019375,Budisa2022}.
In the context of multiscale methods, an application of 
LM to mixed-dimensional methods for a Localized Orthogonal Decomposition methods (LOD) has been recently
proposed in  \cite{altmann-verfuert-2021} for the coupling of bulk (2D) and surface (1D) problems, considering the numerical homogenization of the dynamics on the one-dimensional manifold and enforcing the resulting 
interface conditions with LM.

The model proposed in this work builds on the \textit{reduced} Lagrange multipliers framework recently described by Heltai \& Zunino\cite{LHPZ}, which addresses the systematic construction of dimensionality reduction operators that satisfy the inf-sup condition for constraints for scalar diffusion problems. This approach uses weighted extension and restriction operators for slender domains, leveraging suitable projections onto lower-dimensional manifolds to build reduced infinite-dimensional spaces for the Lagrange multiplier, ensuring controlled approximation properties and enhanced computational efficiency.
Namely, under the assumption of slender (cylindrical) vessels with mostly one-dimensional dynamics, the centerline of the vessels is considered as a representative lower dimensional manifold, and the dimension reduction is achieved by approximating the space of Lagrange multipliers on the two-dimensional vessel boundary with a collection of infinite dimensional Fourier modes on the \textit{one-dimensional} center-line of the vessel.

The main contributions of this work is the extension of the theory developed for scalar fields~\cite{LHPZ} to the particular case of vector-valued fields representing local deformations around one-dimensional manifolds and of multiscale elasticity. This generalization is non-trivial and requires a careful handling of the boundary condition when considering the 
reduced-dimensional space. We focus on the case of axis-symmetric displacement
imposed on the surrounding tissue by a vessel expansion or contraction, which is the
setting of interest when the boundary condition comes from a one-dimensional
model. 
We discuss the reduced-order formulation from the theoretical and practical points of view, and we show how to design a reduced-dimensional space of Lagrange multipliers that can be used to enforce local coupling conditions in a computationally efficient way preserving the stability of the overall formulation at the continous level.
A further contribution is the validation of the proposed framework for the 
modeling of in silico modeling of vascular tissues and for the estimation of 
effective properties. 
Up to our knowledge, this is the first model that
allow for handling dynamics of one-dimensional vessels into
three-dimensional elastic models, and it represents therefore
an important preliminary result for the study of numerical upscaling techniques, for the development of 
fully coupled schemes, and for the solution of multiscale inverse problems for the estimation of tissue properties.

The rest of this paper is organized as follows.
Section \ref{sec:model-problem-preliminary} introduces the  mathematical model from an abstract perspective, 
while Section \ref{sec:model} 
describes the considered mixed-dimensional elasticity problem and its variational formulation.
The reduced Lagrange multiplier formulation 
is introduced and analyzed in Section \ref{sec:reduced}. The numerical results are presented and discussed in Sections \ref{sec:numerics}, while Section \ref{sec:conclusions} draws the concluding remarks.

\section{Mathematical Model}
\label{sec:model-problem-preliminary}

To model the dynamics of a vascularized tissue, let us consider a (Lipschitz) domain $\Omega \subseteq \mathbb{R}^d$, 
containing an elastic material and a set 
$V := \cup_{i=1}^m V_i$ of (possibly disconnected) fluid-filled inclusions $V_i$, $i=1,\hdots,m$, whose boundary will be denoted by $\Gamma := \partial V = \cup_{i=1}^m \partial V_i$.
We address a heterogeneous dimensional coupling, where the physics of the solid and of the fluids are represented by models of different dimensions, i.e., 
a three-dimensional elastic model for the tissue matrix:
\begin{subequations}
\label{eq:abstract-model-problem}
\begin{align}
\label{eq:abstract-model-problem-a}
-\Div (\sigma(\ub)) &= \boldsymbol{f} && \text{ in } \Omega \setminus V \\ \label{eq:abstract-model-no-slip}
\ub - \ub_V(\ub, \phi) &= 0 && \text{ on } \Gamma\\
		\label{eq:abstract-model-problem-c}
\ub & = 0 && \text{ on } \partial \Omega
  \end{align}
  \end{subequations}
and a one-dimensional vessel fluid dynamics in the vessel:
\begin{subequations}
\label{eq:abstract-model-problem-fluid}
\begin{align}
\mathcal{F}(\ub, \phi) & = 0 && \text{ in } V\\
		\label{eq:abstract-model-transmission} 
		\sigma(\ub)\cdot \nb + \sigma_V(\ub, \phi) \cdot \nb_V &= 0 && \text{ on } \Gamma,
  \end{align}
\end{subequations}
This work focuses on the numerical
treatment of the elastic problem.
Thus, for clarity of presentation, we will assume in the following that the vessel dynamics is defined by an abstract operator $\mathcal F$
and by a state variable 
$\phi$.
The function $\ub_V$ represents the displacement of the vessels wall (fluid-solid interface), and $\sigma_V$ is the vessels' stress tensor on the interface. 
Both are assumed to be functions of the displacement $\ub$ and of the vessels state $\phi$.
The function $\boldsymbol{f}$ represents the external forces acting on the elastic material.

When coupling the elasticity and the fluid dynamics problems, we first need to understand what are the relevant variables and how they are related to the deformation of the solid matrix. 
In the context of one-dimensional blood flow models~\cite{MullerBlancoPablo}, the vessel wall is usually modeled as an elastic shell that can deform only in the axial direction 
in response to the pressure difference between the inside and outside of the vessel, and by the flow rate.
No information is typically available on the \emph{spatial location} of the vessels' centerline, which, in these models, is not considered. 
When coupling such a dynamic with the surrounding elasticity problem, special care must be taken to ensure that the coupling conditions are consistent with the one-dimensional physical model and that they do not introduce spurious boundary conditions on the solid matrix. In particular, we need to make some assumptions on how the centerline of the vessel is modified. The most reasonable approach is to impose that, on average, the centerline follows the same rigid deformation of the surrounding solid matrix on $\Gamma$, with the addition of a deformation along the normal direction induced by the one-dimensional model.

Although considering only the radial deformation, i.e., $\ub_V = g \nb$, may seem adequate for most needs (e.g., when the elastic matrix is mostly at rest and for small displacements $\ub$), if used directly in ~\eqref{eq:abstract-model-no-slip} it imposes a spurious homogeneous Dirichlet boundary condition that  \textit{pins} the centerline of the vessel to its current position and rotation. In order to avoid this effect, we need to subtract the average rigid motions of the vessel from the displacement field $\ub_V$ used in the local coupling condition.
This approach results into a coupling condition of the form
\begin{equation}
\label{eq:radial-deformation}
\ub_V(\ub, \phi) := C \ub + \delta R(\phi) \nb \qquad \text{ on } \Gamma,
\end{equation}
where $C$ is an operator that averages the rigid motions of field $\ub$ over the cross section of the vessel, and will be formally introduced in  Section~\ref{sec:model}, and
$\delta R$ is a function that describes the radial deformation of the vessel wall, and it is the input coming from the one-dimensional model.

In the simplified scenario of a unidirectional coupling, i.e., when the fluid equations \eqref{eq:abstract-model-problem-c} are solved separately, and their solution used as a boundary condition for the elasticity problem) one can either use the coupling conditions provided by \eqref{eq:abstract-model-transmission} (a Neumann boundary condition for the elasticity problem), as demonstrated  by Heltai \& Caiazzo~\cite{heltai-caiazzo-2019}, or apply those in \eqref{eq:abstract-model-no-slip} (a Dirichlet boundary condition for the elasticity problem). In the former approach, the vessel displacement is disregarded in the vessel dynamics (i.e., $\ub_V(\phi, \ub)=\ub$ is not enforced explicitly), whereas in the latter, one ignores the elastic stresses, and the vessel displacement is used as a Dirichlet boundary condition for the elasticity problem.
A full coupling can still be achieved by iterating over the two problems in a staggered fashion, i.e., by solving the fluid problem with the current elastic deformation, and then solving the elasticity problem with the current fluid variables until convergence. 

We address the solution of problem \eqref{eq:abstract-model-problem} in the framework of the fictitious domain method, where the elastic model is extended across the entire domain $\Omega$ and the condition~\eqref{eq:abstract-model-no-slip} is enforced directly through a Lagrange multiplier, %
seeking the function $\ub_V$ to ensure that the transmission condition \eqref{eq:abstract-model-transmission} is satisfied (at least in a weak sense) on the interface $\Gamma$, eventually through an iterative process.
In the case os slender vessels, further model-order reduction techniques
can be used to simplify the numerical discretization, 
and adopting a one-dimensional flow model becomes particularly advantageous from a computational perspective~\cite{MullerBlancoPablo}. 

\section{The mixed-dimensional elasticity problem}
\label{sec:model}
\subsection{Preliminaries}
Let us consider, for the sake of simplicity, a domain with a single vessel $V$.
In the setting introduced above, let us
introduce a lower-dimensional manifold $\gamma$ (of co-dimension 2) describing
the centerline of $V$, i.e., a line in 3D and a point in 2D (see Figure \ref{fig:domain-sketch}, left).

Let $\Gamma = \partial V$. We assume that there exists a geometrical projector operator 
\begin{equation}
\label{eq:definition-proj}
\begin{aligned}
& \proj :& \Gamma  &\to &&\gamma \\
\end{aligned}
\end{equation}
that maps uniquely the vessel boundary
$\Gamma$ onto the lower dimensional representative domain $\gamma$.
The inverse of the projection \eqref{eq:definition-proj}, 
$\proj^{-1}: \gamma \to \partitionspace(\Gamma)$,
maps a point $s \in \gamma$ on the associated \textit{cross-section}
\begin{equation}
\label{eq:Ds}
D(s) := \proj^{-1}(s) \subset \Gamma.
\end{equation}
In what follows, we assume that, for any $s$, the cross-sections have the same intrinsic dimension, and denote with $|D(s)|$ their intrinsic Hausdorff measure.

\begin{figure}[!h]
\centering   
\includegraphics[width=0.3\textwidth,trim=0cm 0cm 0cm 0cm,clip=true]{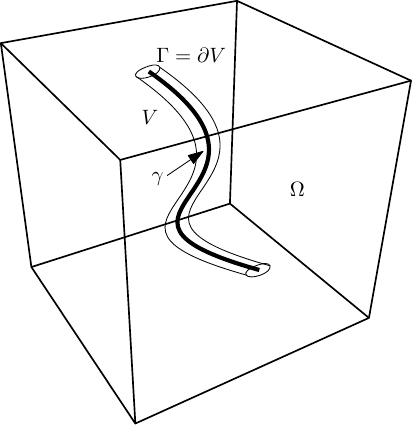}\hspace{1cm}
\includegraphics[width=0.55\textwidth,trim=0cm 0cm 0cm 0cm,clip=true]{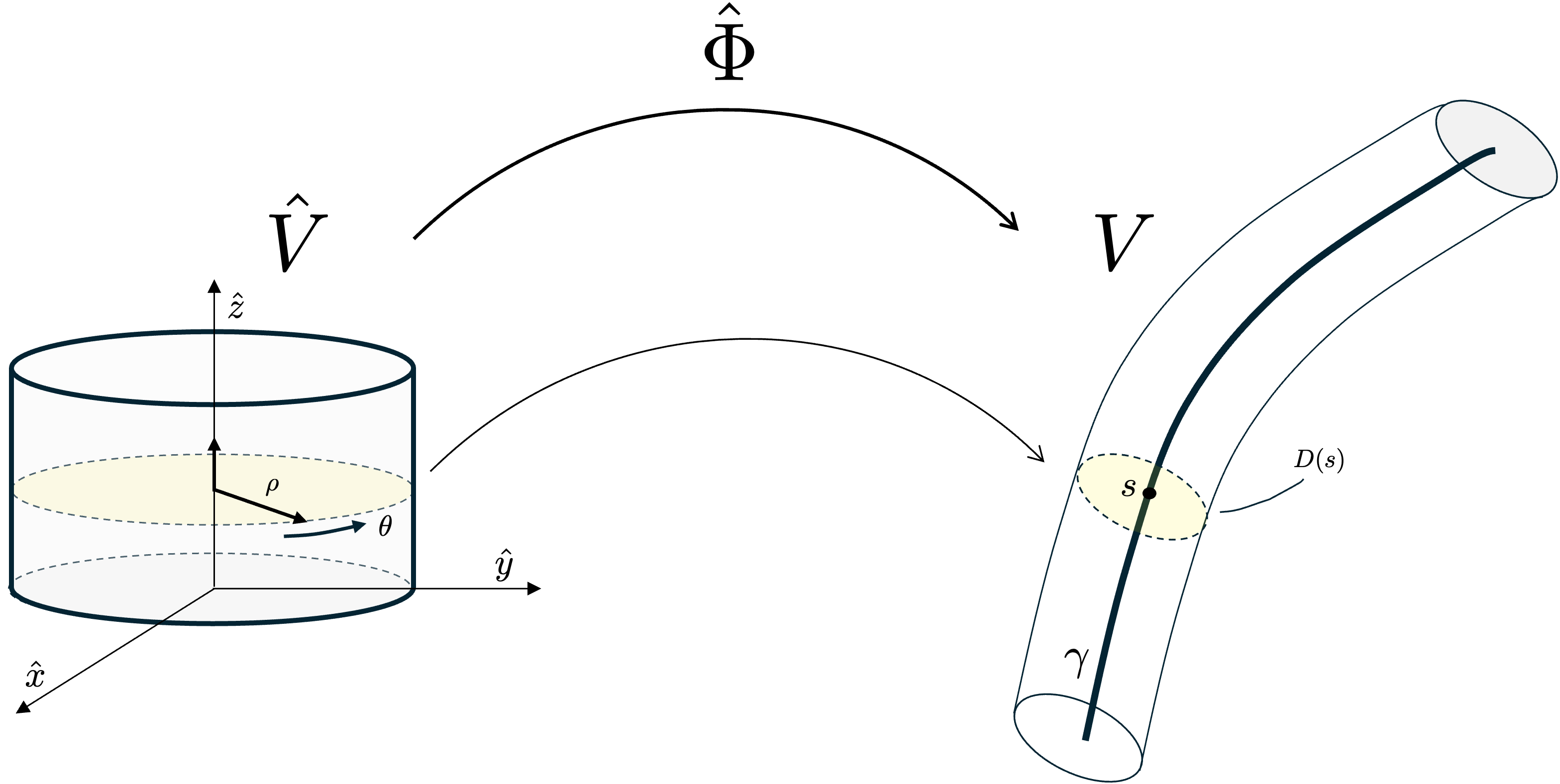}
\caption[Sketch of 3D-1D domain]{Left. Sketch of a sample domain $\Omega$ containing the elastic material and fluid vessel $V$ with boundary $\Gamma \equiv \partial V$. The vessel centerline is denoted by $\gamma$. Right. Each vessel is assumed to be isomorphic to a reference cylinder (unitary radius and height), so that, for each point $s \in \gamma$ on the centerline, it is possible to define a local cylindrical coordinate reference system.}
\label{fig:domain-sketch}
\end{figure}

Furthermore, for the analysis which will be presented in the upcoming sections, we assume that the vessel segment can be mapped into a reference cylinder $\hat V$ of unitary radius and height, whose axis lies parallel to $z$ (see Figure \ref{fig:domain-sketch}, right).
Namely, we assume that there exists an isomorphism 
$\hat \Phi: \hat V \to V$, where 
$\hat V = S^1$ in two dimensions and 
$\hat V = S^1 \times [0,1]$ in three dimensions, such that
\begin{equation}\label{ass:V}
\begin{aligned}
    & \hat \Phi(\hat V) = V \\
    & \hat \Phi \in  C^1\left( \,\overline{\hat V} \,\right), \,\hat\Phi^{-1} \in C^1(V)\\
    & \exists ~ J_{\rm min},J_{\rm max} \mid\,
	0 < J_{\rm min} \leq \text{det}
	\left(
	\nabla \Phi(\hat {\mathbf x})
	\right) \leq J_{\rm max}, \; \forall \hat{\mathbf x} \in \hat V\,.
\end{aligned}
\end{equation}
Moreover, we assume that the preimage 
of the centerline $\gamma$ coincides with the
$z$ axis, i.e.
\begin{equation}\label{ass:hat_gamma}
\hat \gamma = \hat \Phi^{-1}(\gamma)
= \begin{cases}
\{(0,0)\} & d=2 \\
    \{0\} \times \{ 0 \} \times [0,1] & d=3\,.
\end{cases}
\end{equation}
The above assumptions allow to introduce for
each centerline point $s \in \gamma$, a local cylindrical coordinate reference system $(\rho(s),\theta(s),s)$, by considering the preimage on the reference cylinder.

We introduce the following notations for standard Sobolev spaces:
\begin{align*}
	&\spaceV \equiv H^1_0(\Omega)^d, & &\spaceV' \equiv H^{-1}(\Omega)^d,\\
	&\spaceQ \equiv H^{-\frac12}(\Gamma)^d, 
	& &\spaceQ' \equiv H^{\frac12}(\Gamma)^d,\\
	&\spaceW \equiv H^{\frac12}(\gamma), & &\spaceW' \equiv H^{-\frac12}(\gamma)\,,
\end{align*}
as well as the classical trace operator $\traceOp: \spaceV\mapsto \spaceQ'$.

In view of \eqref{eq:abstract-model-problem} and of the condition \eqref{eq:radial-deformation}, we consider the following linear elasticity problem
\begin{subequations}
	\label{eq:model-problem}
	\begin{align}
		\label{eq:model-problem-a}
		-\Div (\sigma(\ub)) &= \boldsymbol{f} && \text{ in } \Omega\setminus V\\
		\label{eq:model-problem-b}
		\ub &= 0 && \text{ on } \partial \Omega \\
		\label{eq:model-problem-c}
		\ub - C\ub
		&= \boldsymbol{g} && \text{ on } \Gamma = \partial V \,,
	\end{align}
\end{subequations}
where the Cauchy stress tensor $\sigma$ is  the usual linear elasticity tensor
\begin{equation}
	\sigma(\ub) := \mu \nabla \ub + \lambda \Div \ub \Id,
\end{equation}
and $\mu$ and $\lambda$ are the first and second Lam\'e parameters. 
The source term $\boldsymbol{g}$
models a coupling with a one-dimensional flow model (radial deformation), and it is assumed to be directed normal to the vessel boundary, i.e., $\boldsymbol{g}:=g_{\nb} \nb = -\delta R(\ub, \phi)\nb$.

To model the
interface condition 
\eqref{eq:radial-deformation}, 
we define the average operator
\begin{multline}
	C \ub(\mathbf x) := \left(\frac1{|D(s)|}\int_{D(s)} \ub \diff{D(s)}\right) + \\
	\left(\frac1{|D(s)|}\int_{D(s)} \ub \cdot \boldsymbol{\tau} \diff{D(s)}\right)  \, \boldsymbol{\tau}(\mathbf x)\qquad, \forall \ub \in C^\infty_c(\Omega)^d\,
	\label{eq:operator_C},
\end{multline}
where, for $\mathbf x \in \Gamma$, $s = \Pi(\mathbf x)$ and $\boldsymbol{\tau}$ is a tangent vector, orthogonal to the normal vector $\nb = \nb(\mathbf x)$ and co-planar with $D(s)$.

\subsection{Variational formulation}
We now extend (continuously) the solution $\ub$ to the entire domain $\Omega$ by
considering the auxiliary, fictitious, problem 
inside $V$:
\begin{equation}\label{eq:el-fictitious}
	-\Div(\sigma(\ub)) = \tilde{\boldsymbol{f}},\; \mbox{in}\;V,
\end{equation}
where $\tilde{\boldsymbol{f}} \in L^2(\Omega)^d$ is an arbitrary extension of $\boldsymbol{f}$ in the
entire $\Omega$ and with boundary conditions that impose continuity of $\ub$ across $\Gamma$.

The corresponding weak formulation of the fictitious formulation of \eqref{eq:model-problem-a} is given by
\begin{equation}\label{eq:weak-extended}
	(\sigma(\ub), \nabla (\vb))_{\Omega} + \duality{\llbracket \sigma(\ub) \rrbracket \cdot \nb, \vb}_{\Gamma} = (f,\vb)_{\Omega},\ \forall \vb \in \spaceV 
\end{equation}
where 
$$
\llbracket \sigma(\ub) \rrbracket := \sigma(\ub)^+-\sigma(\ub)^-
$$
indicates the jump of $\sigma(\ub)$ along the
outgoing normal direction to $\Gamma = \partial V$.

Such procedure is standard in the literature of fictitious
domain methods (see, e.g., ~\cite{Glowinski1999,Glowinski1994}), and allows one to efficiently solve
Dirichlet problems on complex domains, possibly evolving in time, by
embedding them in simpler -- fixed -- domains.
With little abuse of notation, in what follows we will not distinguish  between $\boldsymbol{f}$ and its extension
$\tilde{\boldsymbol{f}}$.

We rewrite \eqref{eq:weak-extended}
imposing condition \eqref{eq:model-problem-c} 
through a Lagrange
multiplier. Namely, we seek $\ub \in \spaceV$, 
$\boldsymbol{\lambda} \in \spaceQ$,  such that
\begin{subequations}
	\label{eq:model-problem-weak}
	\begin{align}
		\label{eq:model-problem-weak-a}
		&(\sigma(\ub), \nabla(\vb))_{\Omega} + 
		\duality{(\traceOp^T - C^T )\boldsymbol{\lambda}, \vb}_{\Gamma}
		= (\boldsymbol{f},\vb)_{\Omega},
		\\
		\label{eq:model-problem-weak-b}
		&\duality{\traceOp\ub - C \ub,\qb}_{\Gamma} = \duality{\boldsymbol{g},\qb}_{\Gamma}\,,
	\end{align}
\end{subequations}
for all $\vb \in \spaceV$ and $\qb \in \spaceQ$.

\begin{remark}
We focus on a static linear elastic material. However, the methodology proposed here can be extended to arbitrarily complex or realistic stress models, as long as the coupling conditions are enforced through Lagrange multipliers.
\end{remark}

Let us define the following operators:

\begin{align}
	\label{eq:defA}
	A &:\spaceV \to \spaceV'
	\\
	\nonumber
	&\duality{A\ub,\vb}:= (\sigma (\ub), \nabla (\vb))_{\Omega}, \;\forall \ub,\vb \in \spaceV 
	\\
	\label{eq:operator_B} 
	B &: \spaceV \to {\spaceQ}'
	\\
	\nonumber
	&\duality{B\ub, \,\qb} := \duality{(\traceOp - C) \ub, \,\qb}_\Gamma, \; \forall \ub \in \spaceV, \forall q \in \spaceQ
	\\
	\label{eq:operator_BT}
	B^T &: \spaceQ \to \spaceV'
	\\
	\nonumber
	&
	\duality{B^T \qb,\, \vb} :=\duality{\left(\traceOp^T  - C^T\right) \qb, \,  \vb}_\Gamma, \;\forall \vb \in \spaceV, \forall q \in \spaceQ
\end{align}

Let us now denote the kernels of $B$ and $B^T$, respectively, as
\begin{equation}\label{eq:kerBt-1}
	\begin{aligned}
		\spaceVz := & \ker(B) 
		 = \left\{ \vb \in \spaceV \mid 
		\duality{\mathbf q, \traceOp \vb}_{\Gamma}= \duality{\mathbf q, C \vb}_\Gamma, 
		\forall \qb \in \spaceQ
		\right\}
		,\\
		\spaceQz := &\ker(B^T) 
		 = \left\{ \qb \in \spaceQ \mid 
		\duality{\mathbf q,\traceOp\vb}_{\Gamma}= \duality{\mathbf q, C \vb}_\Gamma, 
		\forall \vb \in \spaceV
		\right\}
		\,.
	\end{aligned}
\end{equation}

In order for the weak formulation \eqref{eq:model-problem-weak} to be
well-posed, the space of the Lagrange multipliers shall be restricted to
$\spaceQo := \spaceQ \setminus \spaceQz$. Using the notations \eqref{eq:defA},
\eqref{eq:operator_B}, and \eqref{eq:operator_BT}, we thus consider the
following problem:

\noindent Let $f\in\spaceV',\,g\in {(\spaceQo)}'$. Find $u\in\spaceV,\,\boldsymbol{\lambda}\in\spaceQo  $ such that
\begin{subequations}\label{eq:model-problem-weak-abstract}
	\begin{align}
		\label{eq:model-problem-weak-abstract-a}
		&\duality{A\ub,\vb} + \duality{B^T \boldsymbol{\lambda}, \vb} = \duality{\boldsymbol{f},\vb} && \forall \vb \in \spaceV
		\\ \label{eq:model-problem-weak-abstract-b}
		&\duality{B\ub,\qb} = \duality{\boldsymbol{g},\qb} && \forall \qb \in \spaceQo.
	\end{align}
\end{subequations}

Problem \eqref{eq:model-problem-weak-abstract} can be seen as a generalization
of the framework recently described by Heltai \& Zunino\cite{LHPZ} to the case
of a vector-valued problem. The well-posedness of the formulation 
considering a suitable subspace of $\spaceQo$ for the Lagrange multipliers,
will
be addressed in its reduced form of the problem in the upcoming Section 
\ref{sec:reduced}. However, the existence of an inf-sup constant also for 
the full version \eqref{eq:model-problem-weak-abstract} can be shown following
the proof in the scalar case \cite{LHPZ}.

\begin{corollary}[Stresses evaluation]
	From the extended problem \eqref{eq:weak-extended} 
	and from \eqref{eq:model-problem-weak},
	it follows that the Lagrange multiplier satisfies
	\begin{equation}\label{eq:lagrange-multiplier}
		\duality{B^T \boldsymbol{\lambda},\vb}=\duality{(\traceOp^T-C^T)\boldsymbol{\lambda}, \vb} =
		\duality{\llbracket \sigma \rrbracket \cdot \nb, \vb},\; \forall \vb \in \spaceV\,.
	\end{equation}
\end{corollary}
The relation \eqref{eq:lagrange-multiplier} can be used to recover the solid stresses 
on the fluid boundary from the Lagrange multiplier.

\section{Reduced Lagrange multiplier formulation}
\label{sec:reduced}

In this Section, we apply to problem
\eqref{eq:model-problem-weak-abstract} 
the recently introduced reduced Lagrange multiplier framework\cite{LHPZ}, adapting the original
formulation to the case of linear elasticity in $\mathbb R^d$
and taking into account the boundary
condition \eqref{eq:model-problem-c}.
The approach is based on constructing a lower-dimensional space of Lagrange multipliers adapted to 
the geometrical and physical settings of the problem, exploiting the characteristic of how inclusions are modeled and handled.

\subsection{Reduced basis functions}\label{ssec:reduced_basis}

For a given finite set of indices $I \subset \mathbb N$, let us consider a set of functions
$\varphi_i \in H^1(\Gamma) \cap C^0(\overline{\Gamma})$, for $i \in I$,
satisfying, for each $s \in \gamma$\cite{LHPZ},
\begin{equation}\label{eq:phi-i-phi-j}
	\int_{D(s)} \varphi_i \varphi_j \, dD(s) = 0,\; \mbox{for}\; i\neq j,
\end{equation}
i.e., such that $\varphi_i$ and $\varphi_j$ ($i\neq j$) are orthogonal with 
respect to the standard $L^2$ product in each cross section $D(s)$,
and 
\begin{equation}\label{eq:normphi-D}
	\norm{\varphi_i}_{L^2\left(D(s)\right)} = \sqrt{|D(s)|},\;\forall i\in I\,.
\end{equation}

To define the functions $\varphi_i$ we use assumptions 
\ref{ass:V} and \ref{ass:hat_gamma} which relates the vessel interface
to the boundary of a reference cylinder. 

Namely, we start by considering the following definitions in cylindrical coordinates
\begin{equation}
\begin{aligned}
	\label{eq:def-of-varphii}
\hat \varphi_0 (\rho,\theta) &= 1 \\
	{\hat \varphi}_{2k-1} (\rho,\theta) &= 
     \sqrt{2} %
     \rho^k \cos{\theta k},\\
	{\hat \varphi}_{2k} (\rho,\theta) &= 
    \sqrt{2} %
    \rho^k \sin{\theta k},
\end{aligned}
\end{equation}
for $k=1,2,\hdots$, and define 
\begin{equation}\label{eq:phi_x}
\varphi_i (\mathbf x) = \hat \varphi_i (\rho(\mathbf x), \theta(\mathbf x)), \forall i,
\end{equation}
where $\rho(\mathbf x)$
and $\theta(\mathbf x)$
are the cylindrical coordinates of the point $\mathbf x$ in the reference cylinder.
The above definitions satisfy properties \eqref{eq:phi-i-phi-j} and \eqref{eq:normphi-D}
by construction.

A necessary step to define the dimensionality reduction operator is
the characterization of operators that
relate functional spaces on domains of different dimensions.
For each point $\mathbf x \in \Gamma$, 
let us consider the unit vectors $\be_{\alpha}$, 
$\alpha \in \left\{0,d\right\}$, defined in the local Cartesian reference system in 
$s=\Pi(x)$ (with the $z$-axis defined by the axis of the reference cylinder).
For each $i \in I$, we introduce then the average and extension operators
\begin{subequations}\label{eq:AiEi}
	\begin{align}
		\avg^{\alpha,i} :& \spaceQ' &&\to &&\spaceW
		\\
		\nonumber
		&%
		\qb && \mapsto&& \frac1{|D|}\int_{D} \varphi_i\, \be_{\alpha} \otimes \be_{\alpha}\, \qb \diff{D}
		\\
		\ext^{\alpha,i} :& \spaceW &&\to && \spaceQ'
		\\
		\nonumber
		&%
		\wb &&\mapsto&&  \varphi_i \, \be_{\alpha} \otimes \be_{\alpha} \, \wb \circ \Pi,
	\end{align}
\end{subequations}
where $\otimes$ is the outer product, i.e., $(\be_{\alpha} \otimes \be_{\alpha}) \wb := (\wb\cdot \be_{\alpha})\be_{\alpha}$. 

In the case of a scalar problem, the definitions in \eqref{eq:AiEi} naturally extend those proposed by Heltai \& Zunino\cite{LHPZ}.
The boundedness and linearity of $\ext^{\alpha, i}$ and $\avg^{\alpha, i}$ can be also directly derived from the scalar case\cite{LHPZ,Kuchta2021}.
Moreover, from the definitions \eqref{eq:AiEi}, and using \eqref{eq:phi-i-phi-j} and \eqref{eq:normphi-D}
one obtains the orthogonality of the operators
\begin{equation}\label{eq:AiEj_orthogonality}
	\avg^{\alpha,i} \ext^{\beta,j} \wb
	= \delta^{\alpha \beta}\delta^{ij} \, \be_{\alpha} \otimes \be_{\beta} \wb, \qquad \forall \wb \in \spaceW ,
\end{equation}
where $\delta^{ij}$ denotes the Kronecker delta.

Using \eqref{eq:def-of-varphii}, \eqref{eq:phi_x}, and \eqref{eq:AiEi} for each $s \in \gamma$, and working in the local reference system, we can identify the normal and tangential vectors in each point $\boldsymbol{x} \in D(s)$ as 
\begin{align}\label{eq:n-tau}
\boldsymbol{n}(\boldsymbol{x}) & = \rho(\mathbf x)\cos{\theta(\mathbf x)} \,
\be_0  + \rho(\mathbf x) \sin{\theta(\mathbf x)} \,\be_1 =
\left(\ext^{0,1} \underline{\eta} + \ext^{1,2} \underline{\eta} \right), \\
\boldsymbol{\tau} (\boldsymbol{x}) & = \ext^{0,2} \underline{\eta} - \ext^{1,1} \underline{\eta}\,,
\end{align}
where the function $\underline{\eta}$ is a constant vector, taking
the value $\frac1{\sqrt 2}$ along the components
$\be_0$ and $\be_1$.

Consequently, the operator $C$ defining the
boundary condition in \eqref{eq:operator_C} can be expressed in terms of average and extensions operators as
\begin{multline}\label{eq:definitionC}
	C \ub = \underbrace{\ext^{0,0} \avg^{0,0} (\traceOp \ub) +\ext^{1,0} \avg^{1,0} (\traceOp \ub) +
		\ext^{2,0} \avg^{2,0} (\traceOp \ub)}_{\frac1{|D(s)|}\int_{D(s)} \ub \diff{D(s)}} 
    + \\ 
		\underbrace{\ext^{0,2} \avg^{0,2} (\traceOp \ub)
    -\ext^{1,1} \avg^{1,1} (\traceOp \ub)}_{\left(\frac1{|D(s)|}\int_{D(s)} \ub \cdot \boldsymbol{\tau} \diff{D(s)}\right)  \, \boldsymbol{\tau}(\mathbf x)},
\end{multline}
where we also used the orthogonality properties
\eqref{eq:AiEj_orthogonality}.

Let us now consider, for a given integer $N \geq 2$, 
the following space
\begin{multline}\label{eq:WN-def}
\spaceWLagr :=
 \{ \{\wb_i\}_{i = 1}^{i = N}, \wb_i \in [H^{1/2}(\gamma)]^d, \, \forall i =1, \hdots, N \, \\
 \text{ s.t. }  [\wb_1]_y = 0 \text{ and } [\wb_2]_x = 0 \text{ and } [\wb_2]_z = [\wb_1]_z = 0 \} \subset \spaceW^N,
\end{multline}
i.e., a subspace of $\spaceW^N$with an additional constraint motivated by the relation \eqref{eq:definitionC}.
According to \eqref{eq:WN-def}, the space 
$\spaceWLagr$ can be seen as a tensor product of $d^{(N)}:= d(N-2)+2$ copies of a functional space on $\gamma$.

The goal is to define a reduced formulation of 
problem \eqref{eq:model-problem-weak-abstract}, i.e., via a \textit{reduction} operator
\[
R: \spaceQ' \to (\spaceWLagr)' ,
\]
from the \textit{full} space $\spaceQ'$ of functions
defined on the inclusion boundary,
onto the \textit{reduced} space $(\spaceWLagr)'$.
We define $R$ through the adjoint operator 
\begin{equation}\label{eq:RT}
	\begin{aligned}
		R^T: &\spaceWLagr && \to && \spaceQ
		\\
		& \left( \left[\begin{array}{c}
		w_1\\ 0 \\0
	\end{array}
	\right],
    \left[\begin{array}{c} 0 \\ w_2 \\0
	\end{array}
	\right],\wb_3, \hdots,\wb_N \right) && \mapsto && 
		\ext^{0,1} \wb_1 + \ext^{1,2} \wb_2 \\
		& && && + \sum_{i=3}^N \left(\ext^{0,i} \wb_i + \ext^{1,i} \wb_i\right)
		\\
		& && && \in \spaceQ'  \subset \spaceQ.
	\end{aligned} 
\end{equation}

The reduced formulation of 
problem \eqref{eq:model-problem-weak-abstract} can be now written as: 
given $f\in\spaceV',\,\boldsymbol{g}\in\spaceQ'$ find $\ub\in\spaceV,\,\boldsymbol{\Lambda}\in\spaceWLagr$ such that
\begin{subequations}\label{eq:model-problem-weak-abstract-reduced}
	\begin{align}
		\label{eq:model-problem-weak-abstract-reduced-a}
		&\duality{A\ub,\vb} + \duality{B^T R ^T \boldsymbol{\Lambda}, \vb} = \duality{\boldsymbol{f},\vb} && \forall \vb \in \spaceV\,,
		\\
		\label{eq:model-problem-weak-abstract-reduced-b}
		&\duality{R B\ub,\overline{\wb}} = \duality{R  \boldsymbol{g},\overline{\wb}} && \forall \overline{\wb} \in \spaceWLagr\,.
	\end{align}
\end{subequations}

\begin{remark}
    The reduction only acts on the inclusion and the boundary conditions (the operator $B$ in \eqref{eq:model-problem-weak-abstract}), leaving unchanged the elasticity model in the
tissue domain. 

\end{remark}
\subsection{Stability analysis}

The variational formulation \eqref{eq:model-problem-weak-abstract-reduced} is well-posed only if the space of Lagrange multipliers is properly chosen. 
The advantage of the framework introduced in Section \ref{ssec:reduced_basis} and of the operator defined in \eqref{eq:RT} is not only related to the dimensional reduction, but also to the fact that 
the resulting lower-dimensional space naturally ensures the well-posedness.

\begin{lemma}\label{lemma:RT}
The operator $R^T: \spaceWLagr \to \spaceQ$
 is injective. Moreover, it holds
\begin{equation}\label{eq:CuRT}
		\duality{R^T\overline{\wb},C \vb}_{\Gamma} = 0, \; \forall ~ \overline{\wb} \in \spaceWLagr,\; \forall~  \vb \in \spaceV 
	\end{equation}
\end{lemma}
\begin{proof}
The injectivity follows from the orthonormality of $\varphi_1,\hdots,\varphi_N$  \eqref{eq:phi-i-phi-j}, since the image of $R^T$ is a sum of
extension operators $\ext^{\alpha,j}$ for different values of $j$.

We observe that  the composition of the extension and the average operator,
for the same index $i \geq 0$, corresponds to the projection
on the space generated, along the selected component, by the function $\{\varphi_i \}$, i.e.,
\begin{equation}\label{eq:AiEj_projection}
	\ext^{\alpha, i} \avg^{\alpha, i} \qb \in \text{Span}\{\varphi_i\boldsymbol{e}_{\alpha} \} \subset \spaceQ' \qquad \forall \qb \in \spaceQ'.
	\end{equation}
From the definition of the operator C \eqref{eq:definitionC} and from \eqref{eq:AiEj_projection} it follows that, 
for any $\vb \in \spaceV$, $C \vb$ is a projection on 
a space orthogonal to $R^T\left(\spaceWLagr\right)$, 
yielding \eqref{eq:CuRT}.
 
\end{proof}

The previous results ensures the well-posedness of the reduced formulation.
\begin{theorem}[Well-posedness of the reduced problem] \label{theo:infsup-RinfB}
	Let assume that the assumptions \ref{ass:V} are satisfied, and let $N \geq 2$ be given.
Then, the operator $R B: \spaceV \mapsto \spaceWLagr$ satisfies the inf-sup condition, i.e., there exists a positive real number $\beta_R>0$ such that
	\begin{equation}
		\infsup{\overline{\wb} \in\,\spaceWLagr}{\vb \in\spaceV} \frac{\duality{R B\vb, q}}{\norm{\vb}_{\spaceV} \norm{\overline{\wb}}_{\spaceWLagr}} \geq \beta_R > 0.
	\end{equation}
\end{theorem}

\begin{proof}
Firstly, let us observe that the operator $R B$  is bounded as it is a combination of trace operator and linear bounded operators\cite{LHPZ}.
From assumptions \ref{ass:V}, it follows also that the operator $(RB)^T$ is linear, continuous, and bounded.
	
From standard saddle-point theory~\cite{Boffi2013}, a
sufficient condition for
the well-posedness is that $(RB)^T$ is injective.

Let $\overline{\wb} \in \spaceWLagr$ such that $B^T R^T \wb = 0$.
It follows that $R^T \overline{\wb} \in \ker (B^T)$, i.e.

$$
\duality{R^T \overline{\wb},\traceOp\vb}_{\Gamma}= \duality{R^T \overline{\wb}, C \vb}_\Gamma, \forall \vb \in \spaceV. $$

From the properties of $R^T$ (Lemma \ref{lemma:RT}) it follows that
$$
\duality{R^T \overline{\wb},\traceOp\vb}_{\Gamma}
= 0,\; \forall \vb \in \spaceV,
$$
and hence $R^T \overline{\wb} = 0$.
Since $R^T$ is injective on
$\spaceWLagr$ (Lemma \ref{lemma:RT}), we obtain that
$\overline{\wb} = 0$. 

\end{proof}

\subsection{Axis-symmetric deformation}
\label{sec:axis-sym}

In this Section, we discuss more in details the implications of considering a coupling conditions modeling an axis-symmetric deformation 
of the vessel wall, i.e.,  resulting in a source term in \eqref{eq:model-problem} of the form
\begin{equation}\label{eq:g_n}
	\boldsymbol{g}(\mathbf x) = g_{\gamma}\left(\Pi(\mathbf x)\right) \nb(\mathbf x),
\end{equation}
directed along the normal $\nb$ to the interface $\Gamma$, where $g_{\gamma}: \gamma \to \mathbb R$ 
denotes the inflation or deflation of the vessel.

Based on the notation used in equation \eqref{eq:RT}, let us decompose the reduced space as
\begin{equation}
\label{eq:WN-product}
\spaceWLagr = \boldsymbol{W}_{\gamma}^{(1,2)} \times 
\boldsymbol{W}_{\gamma}^{(3,N)}\,,
\end{equation}
where $\boldsymbol{W}_{\gamma}^{(1,2)} := H^{1/2}(\gamma) \times H^{1/2}(\gamma)$ specifies
the degrees of freedom of the first two components, 
and
$\boldsymbol{W}_{\gamma}^{(3,N)} := \Pi_{i=3}^N [H^{1/2}(\gamma)]^d$ refers to the components from $3$ to $N$.
In line with \eqref{eq:WN-product},  let us denote with $\overline{\wb} = \overline{\wb}_{(1,2)} + \overline{\wb}_{(3,N)}$ 
each $\overline{\wb} \in \spaceWLagr$
as sum of two components.

Let us recall that the normal vector to the interface can be expressed as a combination of the basis functions
$\varphi_1$ and $\varphi_2$ (see equation \eqref{eq:n-tau}) and belongs to the subspace $\boldsymbol{W}_{\gamma}^{(1,2)} \times \left\{0\right\}$ of $\spaceWLagr$.
According to the definition of $R^T$, for a source term
normal to the interface it holds
$$
\duality{R\boldsymbol{g},\overline{\wb}}
= \duality{\boldsymbol{g},R^T
\overline{\wb}} = 
\duality{\boldsymbol{g},R^T
\overline{\wb}_{(1,2)}}\,,
\forall
\overline{\wb} \in \spaceWLagr\,.
$$

The latter allows to reformulate 
equation 
\eqref{eq:model-problem-weak-abstract-reduced-b} of the reduced Lagrange multipliers problem component-wise as
$$
\begin{aligned}
\duality{R B \ub,\overline{\wb}_{(1,2)}} & = \duality{R\boldsymbol{g},
\overline{\wb}_{(1,2)}}\\
\duality{R B \ub,\overline{\wb}_{(3,N)}} & = 0\,.
\end{aligned}
$$
This observation 
motivates us to consider,
for the case of axis-symmetric source terms, 
only the case $N=2$, since the contribution
to the solution of the
additional degrees of freedom
introduced choosing $N>2$ is expected to be small, in particular, 
depending on how much the solution
$\ub$ differs, locally, from a
axis-symmetric (normal) displacement field.
The numerical test presented in Section \ref{ssec:higher-order} 
will focus on the error committed when omitting the extra modes.

\section{Numerical results}
\label{sec:numerics}
This section is dedicated to the numerical validation of
the reduced Lagrange multiplier approach described in Section \ref{sec:reduced}
considering different benchmarks.

First, we monitor the converge rate of the method on a simplified case where the analytical solution is known (Section \ref{ssec:2d-model-problem}), and investigate the role of modes in relation to the accuracy of the method (Section \ref{ssec:higher-order}). Next, we study how the presence of the inclusions affects the elastic properties of the domain, considering both two- and three-dimensional setups
(Section \ref{ssec:effective-mat}).
The numerical simulations have been obtained using the finite element library \textsf{deal.II} \cite{dealII94}. 

\subsection{Two-dimensional axis-symmetric problem}\label{ssec:2d-model-problem}
The first example considers a circular domain of 
radius $R$, containing a single circular inclusion of radius $r_i \ll R$ located at its center.
Imposing homogeneous Dirichlet boundary condition on the
outer radius and at a given normal displacement $\bar{u}$ at the inclusion boundary, problem \eqref{eq:model-problem} admits the analytical solution in polar coordinates 
\cite{heltai-caiazzo-2019}
\begin{equation}
	\label{eq:analytic}
	u_r(r,\theta) = c_1(r_i, R, \bar{u}) \left(\frac{R^2}{r} - r\right),\;
    u_{\theta}(r,\theta) = 0\,.
\end{equation}
with $c_1 =  \dfrac{ r_i\bar{u}R^2}{R^2 - r_i^2}$.
The analytical solution for the Lagrange multiplier is 
$\Lambda = 12 \pi \frac{\bar{u}}{R^2 - r_i^2}$.
The formulation \eqref{eq:model-problem-weak-abstract-reduced} has been solved with piecewise linear finite elements for the displacement and with different values of the dimension $N$ of the reduced Lagrange multiplier space.
\begin{figure}[!ht]
	\centering
	\includegraphics[width=.8\textwidth, page=7]{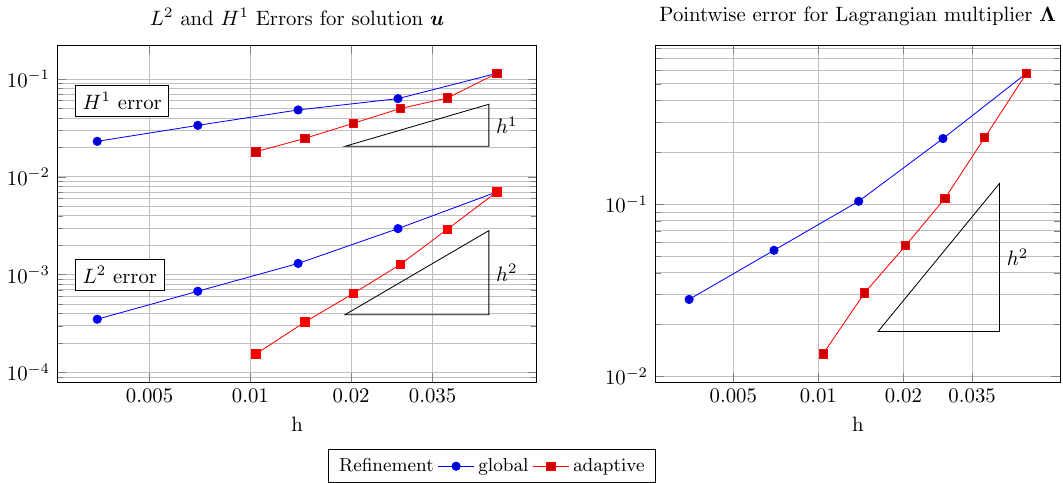}\hspace{1cm}
	\caption[Test case 1: Convergence rate]{
		Left: Axis-symmetric problem with $R=1$, $r_i = 0.2$, and $\bar{u}=0.1$.
		Convergence rates for the case $N=2$ with adaptive and global refinements. Right: Convergence rate of the reduced Lagrange multiplier. The discretization size ($x$-axis) is defined as $h := (\#\text{dofs/d})^{-d}$
	}
	\label{fig:err_test1}
\end{figure}
The results of the convergence study are
summarized in Figure \ref{fig:err_test1} (left).
Using an adaptive mesh refinement, we obtain second order 
convergence in $L^2$ and first order in $H^1$, 
consistent with the expectations of the immersed method (see \cite{LHPZ}). 
Using an adaptive refinement
we obtain a second order convergence also for the Lagrange multiplier
(Figure \ref{fig:err_test1}, right).

\subsection{Effect of the higher-order modes}\label{ssec:higher-order}
As observed in Section \ref{sec:axis-sym}, 
imposing a normal source on the
inclusion boundary, whose magnitude only depends on the position along
the representative manifold $\gamma$, justifies 
the approximation with $N=2$ (and a two-dimensional reduced-order space).
However, in general, the role of higher-order modes is subjective to the case analyzed, especially in the case of multiple inclusions where the 
the overall domain is no longer axis-symmetric, as underlined for the equivalent scalar problem in \cite{LHPZ}, and the effect of the local boundary conditions might superimpose an yield to non-negligible component
also on the tangential directions.

The purpose of this example is to investigate
the impact of the choice $N=2$ in a simple setting.
For this purpose,  we consider $m=3$ inclusions in a squared domain, and
compute the resulting stresses via Equation \eqref{eq:lagrange-multiplier}
(i.e., as a function of the reduced Lagrange multipliers) for different
geometrical parameters (radii of the inclusions) and for different dimensions of the reduced Lagrangian multipliers space.

Let $N\geq 2$ be fixed and equal for all the inclusions, let us denote with 
$\boldsymbol{W}_{\gamma_i}^{(N)}$
the reduced space for the $i$-th inclusion spanned by $N$ modes, and with
\[
\boldsymbol{W}^{(N)} := \Pi_{i=1}^m \boldsymbol{W}_{\gamma_i}^{(N)}
\]
the space of Lagrange multipliers for all inclusions.

We focus on a two-dimensional spatial setting ($d=2$).
In this case, inclusions are represented by points and 
$\boldsymbol{W}^{(N)}$ has dimension $m \times(2N-2)$.
As done in \eqref{eq:RT}, let us denote an element 
$\bLam^{(N,i)} \in \boldsymbol{W}_{\gamma_i}^{(N)}$ as
\begin{equation}
\bLam^{(N,i)} = 
\left[
		 \left[\begin{array}{c}
		\Lambda^{(N,i)}_{1,x} \\[0.5em] 0
	\end{array}\right], \;
    \left[\begin{array}{c}
		 0 \\[0.5em] \Lambda^{(N,i)}_{2,y}
	\end{array}\right],\left[\begin{array}{c}
		\Lambda^{(N,i)}_{3,x} \\[0.5em] \Lambda^{(N,i)}_{3,y}
	\end{array}\right],\hdots,\left[\begin{array}{c}
		\Lambda^{(N,i)}_{N,x} \\[0.5em] \Lambda^{(N,i)}_{N,y}
	\end{array}\right]
	\right]
\label{eq:LambdaComponents}	
\end{equation}
in terms of the degrees of freedom corresponding to the Lagrange multipliers. 

To assess quantitatively the
error resulting from using only two modes ($N=2$) on each inclusion, 
we introduce the quantity 
\[
\text{e}_{N,i} = \frac{\|\bLam^{(N,i)}\|_{l2}^2 - \left(\left|\Lambda^{(N,i)}_{1,x}\right|^2 + \left|\Lambda^{(N,i)}_{2,y}\right|^2\right)}{\|\bLam^{(N,i)}\|_{l2}^2}\,,
\]
i.e., representing the magnitude of the components of 
$\bLam^{(N,i)}$ of index greater than two.
The considered numerical example, with $m=3$, is depicted together with
the numerical solution in a particular configuration in Figure \ref{fig:subsec_modes_example}.

\begin{figure}[ht]
\centering
\begin{tikzpicture}
    \centering
   \node (img) at (0,0){\includegraphics[width=0.5\textwidth]{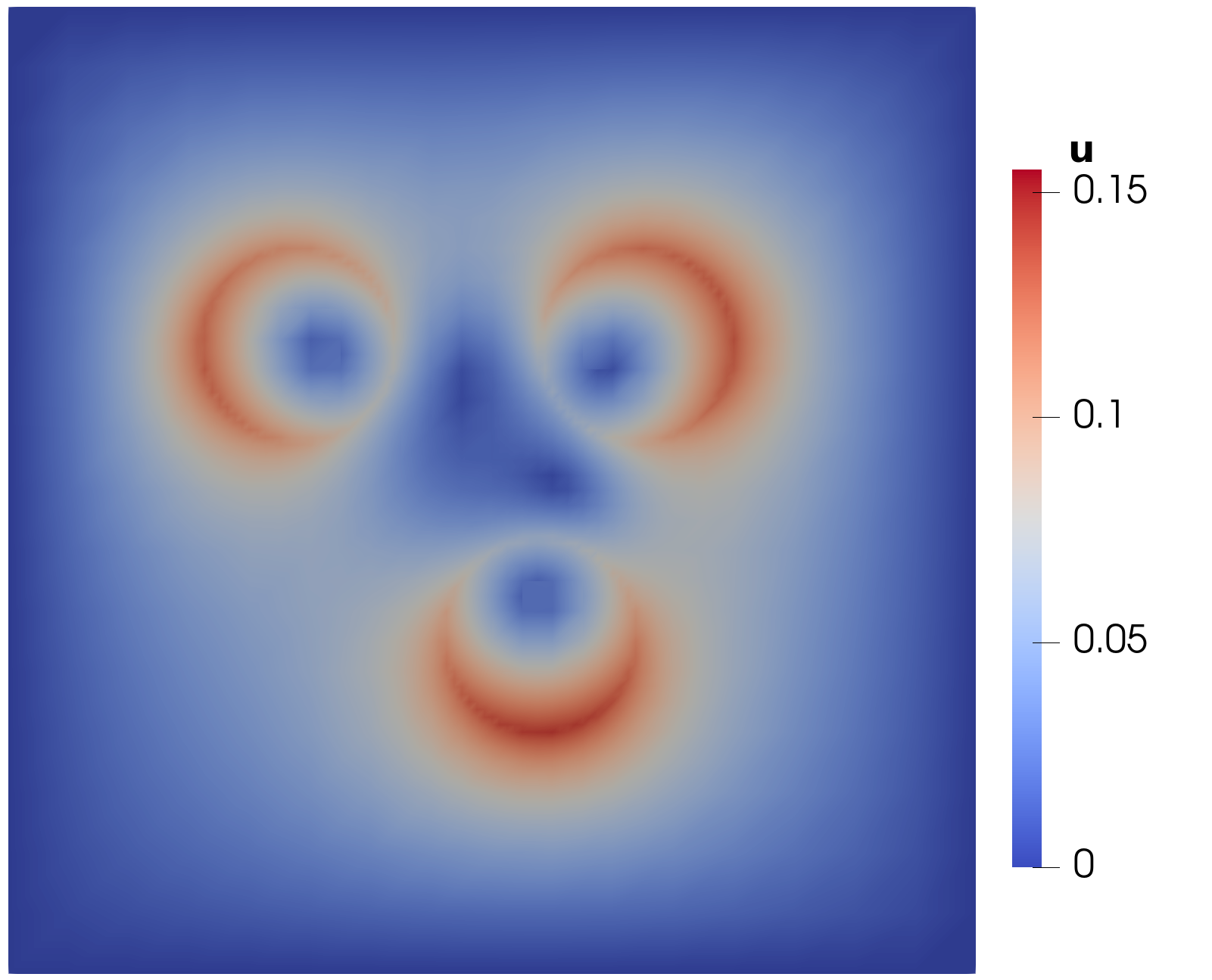}};
   \node [below left,text width=1cm,align=center] at (0.65,1){1};
   \node [below left,text width=1cm,align=center] at (-1.4,1.3){2};
   \node [below left,text width=1cm,align=center] at (0.02,-0.5){3};
   \end{tikzpicture}
    \caption[test case 3]{Example of solution with three inclusions, showing the set up use for the analysis of modes, with inclusions positioned at $(0.3,0.3), (-0.4,0.3),(0.1,-0.3),~r_i= 0.2,~\bar{u}=0.1$, in domain $[-1,1]^2$}
    \label{fig:subsec_modes_example}
\end{figure}

\begin{table}
	\centering
	\pgfplotstableread[col sep=space,comment chars={c}]{./data/test3_lambdas.txt}\DataModesTable
	\pgfplotstabletypeset[
	    col sep=space,
	    string type,
	    create on use/L21/.style={create col/expr={(abs(\thisrow{lambda14})^2)/\thisrow{lambda1norm}}},
	    create on use/L22/.style={create col/expr={(abs(\thisrow{lambda15})^2)/\thisrow{lambda1norm}}},
	    create on use/error/.style={create col/expr={max(1-((abs(\thisrow{lambda14})^2)/\thisrow{lambda1norm} + (abs(\thisrow{lambda15})^2)/\thisrow{lambda1norm}), 0)}},
	    columns/lambda1norm/.style={column name=$\|\boldsymbol{\Lambda}^{(N,2)}\|_{l2}$},
	    columns/L21/.style={column name=$\frac{|\Lambda^{(N,2)}_{1,x}|^2}{\|\boldsymbol{\Lambda}^{(N,2)}\|_{l2}^2}$,dec sep align,
	        preproc/expr={100*##1},
	        postproc cell content/.append code={
	            \ifnum1=\pgfplotstablepartno
	                \pgfkeysalso{@cell content/.add={}{\%}}%
	            \fi
	        },
	    },
	    columns/L22/.style={column name=$\frac{|\Lambda^{(N,2)}_{2,y}|^2}{\|\boldsymbol{\Lambda}^{(N,2)}\|_{l2}^2}$,dec sep align,
	        preproc/expr={100*##1},
	        postproc cell content/.append code={
	            \ifnum1=\pgfplotstablepartno
	                \pgfkeysalso{@cell content/.add={}{\%}}%
	            \fi
	        },
	    },
	    columns/error/.style={column name={$\text{e}_{N,i}$ (\%)},dec sep align,column type/.add={}{|},
	        preproc/expr={100*##1},
	        postproc cell content/.append code={
	            \ifnum1=\pgfplotstablepartno
	                \pgfkeysalso{@cell content/.add={}{\%}}%
	            \fi
	        },
	    },
	    columns/ri/.style={column name=$r_i$},
	    columns/modes/.style={column name=\# Modes},
		every head row/.style={before row=\hline,after row=\hline\hline, column type/.add={}{|}},
	    every last row/.style={after row=\hline},
	    every first column/.style={column type/.add={|}{|}},
	    every last column/.style={column type/.add={}{|}},
	    every nth row={7}{before row=\midrule},
	    columns={ri, modes, lambda1norm, L21, L22, error}
	]\DataModesTable
	\caption{Euclidean $l^2$-Norm of the Lagrange multiplier for the second inclusion (number 2 in Figure \ref{fig:subsec_modes_example}), relative norms of the first two modes, and corresponding relative error for different sizes of inclusions.}
	\label{tab:errors}
	\end{table}

Table \ref{tab:errors} shows the norm of the Lagrange multiplier solution for the second inclusion ($i=2$), 
as a function of the
inclusion radius and of the number of modes (up to $N=8$), comparing the leading order modes 
($\Lambda_1$ and $\Lambda_2$) with the remaining ones.
The results for the other inclusions are very similar and will be omitted from the discussion.

The first two components are largely predominant, and the relative truncation error of the order of $10^{-2}$ (or less).
This validates the choice $N=2$ also in the non-symmetric case, provided
the radii of the inclusions are small enough. In fact, the truncation error decreases when decreasing the size of the inclusion.
This observation is in line with what observed
in the context of 3D-1D models coupled via Neumann boundary conditions\cite{heltai-caiazzo-2019}, in which an hypersingular, 
axis-symmetric, approximation of the immersed interface was used.

Figure \ref{fig:modes_2} depicts in more detail the predominance of the leading modes ($i=1, 2$) the relative decrease of the higher modes ($N=3$ to $N=8$) magnitudes, normalized by to the Euclidean $l^2$ norm of the Lagrange multiplier. 
\begin{figure}
	\centering
	\includegraphics[width=.7\textwidth, page=1]{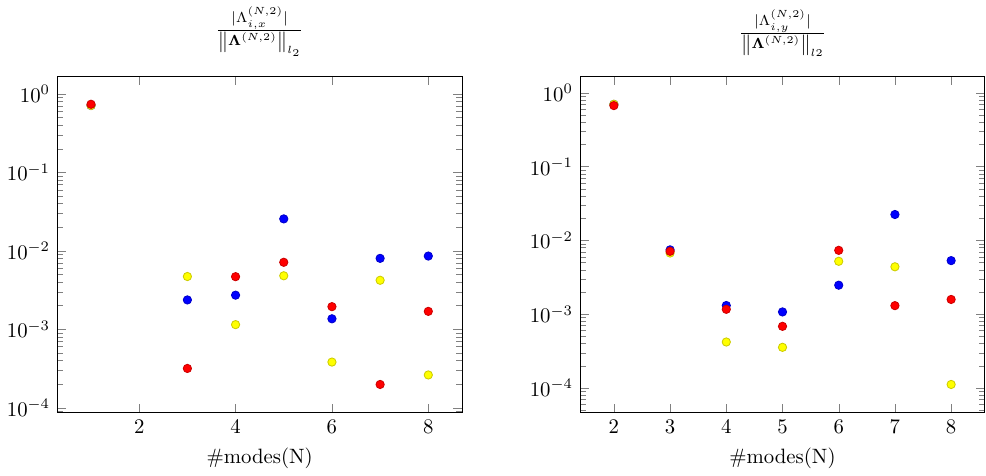}
	\caption[Test case 3]{%
	Components of vector $\bLam$ by $N$, total number of modes used, for $r_i = $ 0.2 (red), 0.1 (yellow), 0.05 (blue). The three highest values in both graphs overlay almost perfectly. Left figure shows the first component ($\Lambda^{(N,2)}_{i,\alpha}, ~ \alpha = x$), i.e. first line of Eq. \eqref{eq:LambdaComponents}, right figure shows $\alpha = y$, the second line of \eqref{eq:LambdaComponents}.}\label{fig:modes_2}
   \end{figure}

\subsection{In-silico modeling of effective material behavior}\label{ssec:effective-mat}
The proposed multiscale model is motivated by applications in the context of tissue imaging, where the data acquired by 
techniques, such as elastography or diffusion-weighted imaging, depend on the underlying physics -- e.g.,  on the interaction of
solid and fluid phases -- but the limited image resolution allows only for effective (macroscale) tissue representations.
Often, these effective descriptions are based on linear elasticity with homogeneous mechanical parameters. However, certain applications
require to better understand how the fluid phase, or the structure of the vasculature, are reflected
in the behavior of the tissue at the macroscale. This is the case, for instance, of the usage of medical imaging to 
characterize the presence of pathological conditions in which fluid conditions play a relevant role, such as hypertension 
(increase in pressure) or tumor growth.
The numerical tests presented in this section are devoted to the usage of the reduced Lagrange multipliers method to bridge this multiscale gap, investigating, through computational models, the influence of fluid microstructures on tissue effective dynamics.

On the one hand, these tests address, from the perspective of mathematical modeling, results recently
presented in the context of tissue elastography concerning the importance of understanding the interplay between 
solid and fluid phases for medical imaging applications in non-invasive diagnostic\cite{JAITNER2025312,safraou-etal-2023,palaniyappan-etal-2016,guo-etal-2020}.
On the other hand, the in silico study aims at providing a first proof of concept for using the reduced
Lagrange multipliers in the context of inverse problems for the estimation of 
effective mechanical parameters.
\subsubsection{Effective material parameters for varying microstructure}

The first example considers a two-dimensional tissue sample
with a fixed volume fraction, but with different distributions of the inclusions in space.
Namely, we consider a fixed number of inclusions within the domain, with the same radii, but creating three different geometrical distributions (see Figure \ref{fig:microstructures}):
\begin{itemize}
	\item[(i)] inclusions placed in a structured array (denote, in what follows, as \textit{structured} setup);
	\item[(ii)] inclusions placed randomly, but  with fixed \textit{microscale} volume fraction, i.e., dividing the domain in boxes, and placing, within each box, an inclusion in a random position (\textit{semi-structured} setup);
	\item[(iii)] inclusions placed fully randomly, but with fixed volume ratio at the \textit{macroscale} (\textit{random} setup); these configurations have been realized removing overlapping inclusions, iteratively adding new inclusions until the fixed total number has been reached, and performing $20$ simulations for each number of inclusions.
\end{itemize}
\begin{figure}[!h]
	\centering
	\includegraphics[page=1, width=0.25\textwidth]{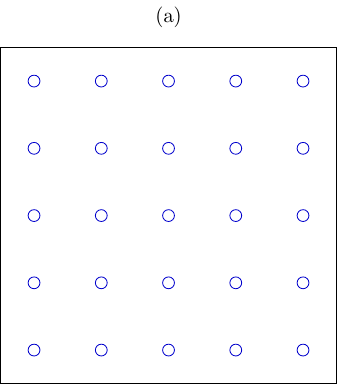} \hspace{0.1cm}
	\includegraphics[page=2, width=0.25\textwidth]{standalonefigures/test45_setup.pdf} \hspace{0.1cm}
	\includegraphics[page=3, width=0.25\textwidth]{standalonefigures/test45_setup.pdf}
	\caption{Example of the different setups used for the modeling and simulation of effective material (with $m=25$ inclusions, $r_i = 0.05$, $v_f \sim 0.05$). Left: inclusions in a structured array. Center: inclusions placed randomly within structured placed boxes (i.e., fixing the porosity in each box). Right: Inclusion placed randomly, removing overlapping ones and fixing the volume fraction.}\label{fig:microstructures}
\end{figure}
In these settings, we study how the behavior of the
\textit{effective} material properties depend on the microstructure configuration, as a function of the Lame's constant $\lambda$, $\mu$ of the solid matrix, of the boundary condition imposed at the inclusion boundaries (the normal deformation),
of the total volume fraction (i.e., of the number of inclusions),
and of the inclusions distribution. 
Of particular interest is the stress at the boundary of the domain induced by the presence of the internal structure. This would result in an added value in the effective mechanical properties of the composite material.

Let us introduce the quantity
$$
F_i^{\alpha} = \int_{\Gamma_{\alpha}} [\sigma(\ub) \nb ]_i,
$$
i.e., the inclusion - induced stress on a
boundary face $\Gamma_{\alpha}$, and on a given direction 
$i \in \left\{1, \hdots,d\right\}$. A sketch of the 
considered numbering of the faces in provided in Figure 
\ref{fig:test45_setup}.
We will refer to $F^{\alpha}_i$ as \textit{tangential stress} if $\nb_{\alpha} \perp i$ and as \textit{normal stresses} otherwise. 
\begin{figure}
	\centering
\includegraphics[page=4]{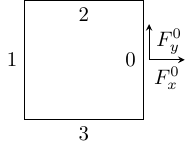}
\hspace{0.5cm}
\includegraphics[page=12]{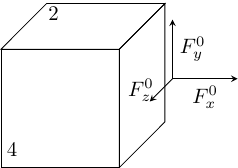}
\caption{Schematic geometry of the 2D and 3D test cases with face numbering.}\label{fig:test45_setup}
\end{figure}
The goal of this test is to highlight that both components of the stress (normal and tangential) need to be adjusted to the presence of the internal structure by adding the inclusion -- induced stress. 

\begin{figure}
	\centering
	\includegraphics[width=.43\textwidth, page=5]{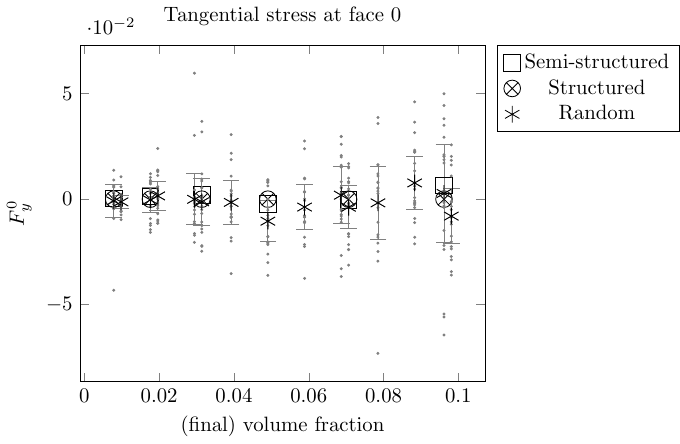}
		\includegraphics[width=.43\textwidth, page=8]{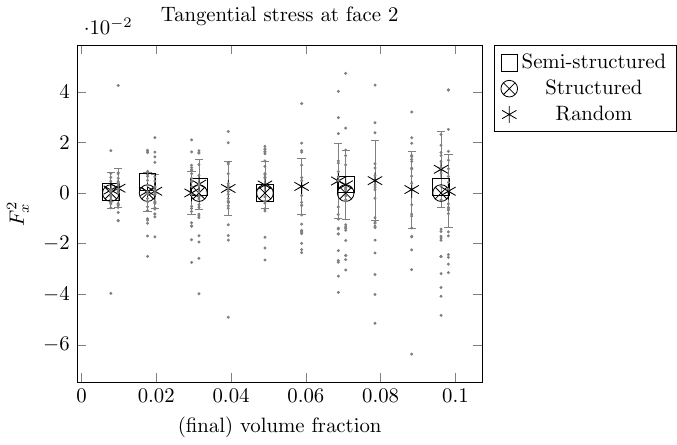}
\caption{Tangential stress, varying the volume fraction and
	for the different geometrical setups. 
	The results for the random configuration show the
outcome of each of the 20 simulations (tiny gray dots), the
    averages, and the standard deviations (gray bars).}
\label{fig:tangential}
\end{figure}
Figure \ref{fig:tangential} highlights that
the effect on the tangential stresses are
mostly due to the breaking of symmetry 
in the random configurations (i.e., with highest values
for the configurations where inclusions are
placed very close to the boundary).
In particular, Figure \ref{fig:40} shows the two cases for $40$ inclusions that produce the highest and lowest value of $F^2_x$.
\begin{figure}[!ht]
	\centering
	\includegraphics[page=5, width=0.25\textwidth]{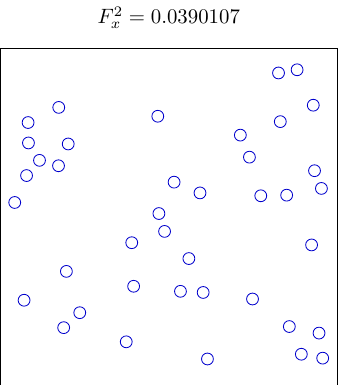} \hspace{0.2cm}
	\includegraphics[page=6, width=0.25\textwidth]{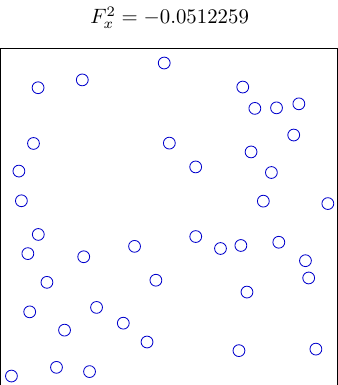}
	\caption{Configurations of 40 inclusions that show the highest and lowest value (among the 20 realization of random arrangement) of Tangential stress.}
	\label{fig:40}
\end{figure}
On the other hand, when looking at the normal stress (Figure \ref{fig:normal}) the volume fraction (i.e., the number of inclusions) has the most significant influence, and the relation with is nearly linear.
\begin{figure}[!ht]
	\centering
	\includegraphics[width=.43\textwidth, page=1]{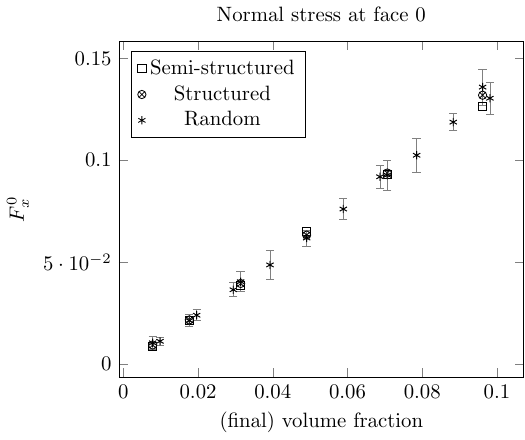}
	\includegraphics[width=.43\textwidth, page=3]{standalonefigures/stresses_test4.pdf}
	\caption{Normal stress, varying the volume fraction and
		for different inclusion distributions. 
		The results for the random configuration show the averages and the standard deviations (grey bars) over $N=20$ simulations}
	\label{fig:normal}
\end{figure}

\subsubsection{In silico evaluation of effective mechanical properties}
The previous examples showed that considering an underlying homogeneous structure for the fluid inclusion results in effective
parameters mostly dependent only on the volume fraction. The purpose of the next set of numerical tests is to investigate
the influence of non uniform inclusion distributions on the effective mechanical behavior.

To this purpose, we divide the tissue in two subdomains
(inner and outer), considering structured distributions with a higher density in the inner part.
Namely, the outer domain contains, in all considered cases, 12 inclusions, whilst the number of inclusions in the inner domain varies from 9 (3$\times$3 array) to 121 (11$\times$11 array).
The resulting arrangements are shown in Figure \ref{fig:microstructuresCentered}. 
All inclusions have the same radius  $r_i = 0.05$ and 
the elastic properties of the solid matrix are set as $\lambda = \mu = 1$.

\begin{figure}[!ht]
	\centering
	\includegraphics[page=7 , width=.19\textwidth]{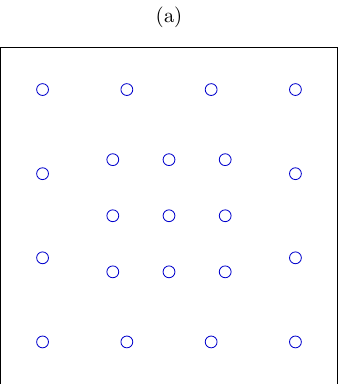}
	\includegraphics[page=8 , width=.19\textwidth]{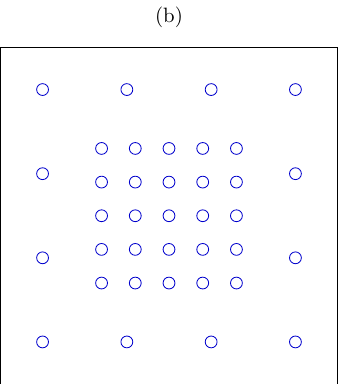}
	\includegraphics[page=9 , width=.19\textwidth]{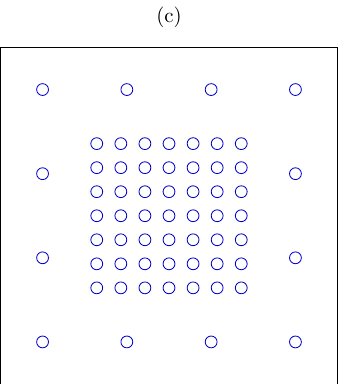}
	\includegraphics[page=10, width=.19\textwidth]{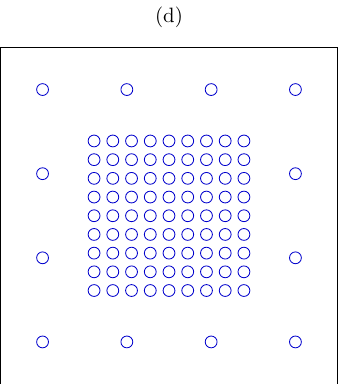}
	\includegraphics[page=11, width=.19\textwidth]{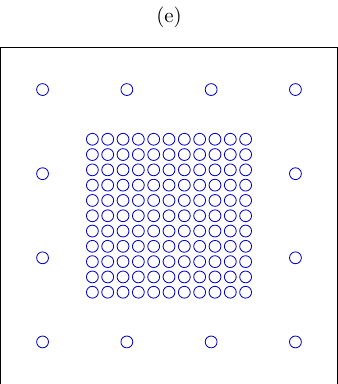}
	\caption{Different setups used for the modeling and simulation of effective material, with $m=21, ~37, ~61,~93$ and $133$ inclusions
		, with radius $r_i = 0.05$.
	}\label{fig:microstructuresCentered}
\end{figure}

For each setup, in which the volume fraction and the number of inclusions are fixed, we impose different values of displacement $\bar{u}$ at the inclusions boundary.

Results for the two-dimensional case are
shown in Figure \ref{fig:test4_core}. 
On the one hand, as already seen in the previous example, the symmetry prevents any significant effect on the tangential stress. On the other hand, the normal inclusion-induced stress increases with the volume fraction and with the inclusions expansion value.
\begin{figure}[!ht]
	\centering
	\includegraphics[height=.3\textwidth, page=10]{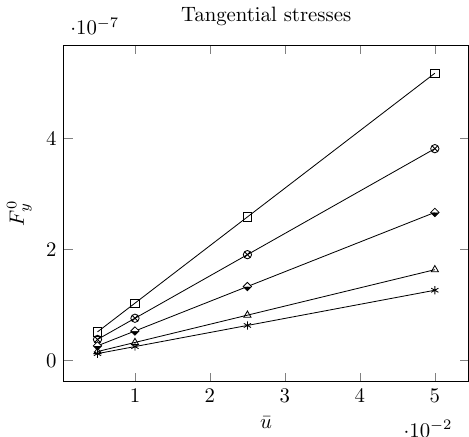}
    \hspace{0.2cm}
    \includegraphics[height=.3\textwidth, page=11]{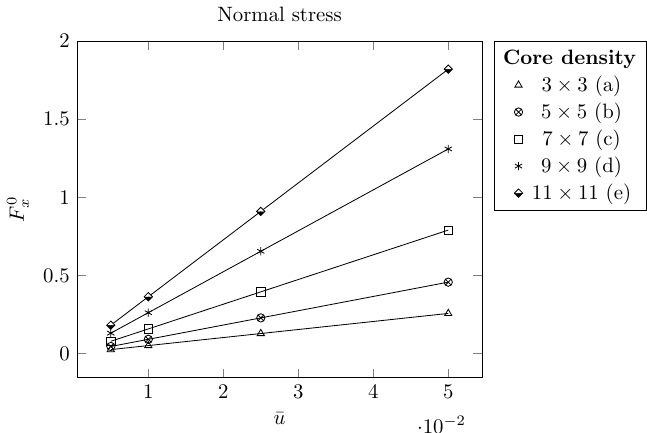}
	\caption{Tangential (left) and normal (right) inclusion-induced stress on face 0, for different values of $\bar{u}$ on the set-up cases in Figure \ref{fig:microstructuresCentered}.}\label{fig:test4_core}
\end{figure}

The similar setup has been also investigated in a three-dimensional setting, considering the inclusions aligned to
the $z$ axis (see Figure \ref{fig:test3D_setup}).
\begin{figure}
	\centering
	\includegraphics[height=.25\textwidth]{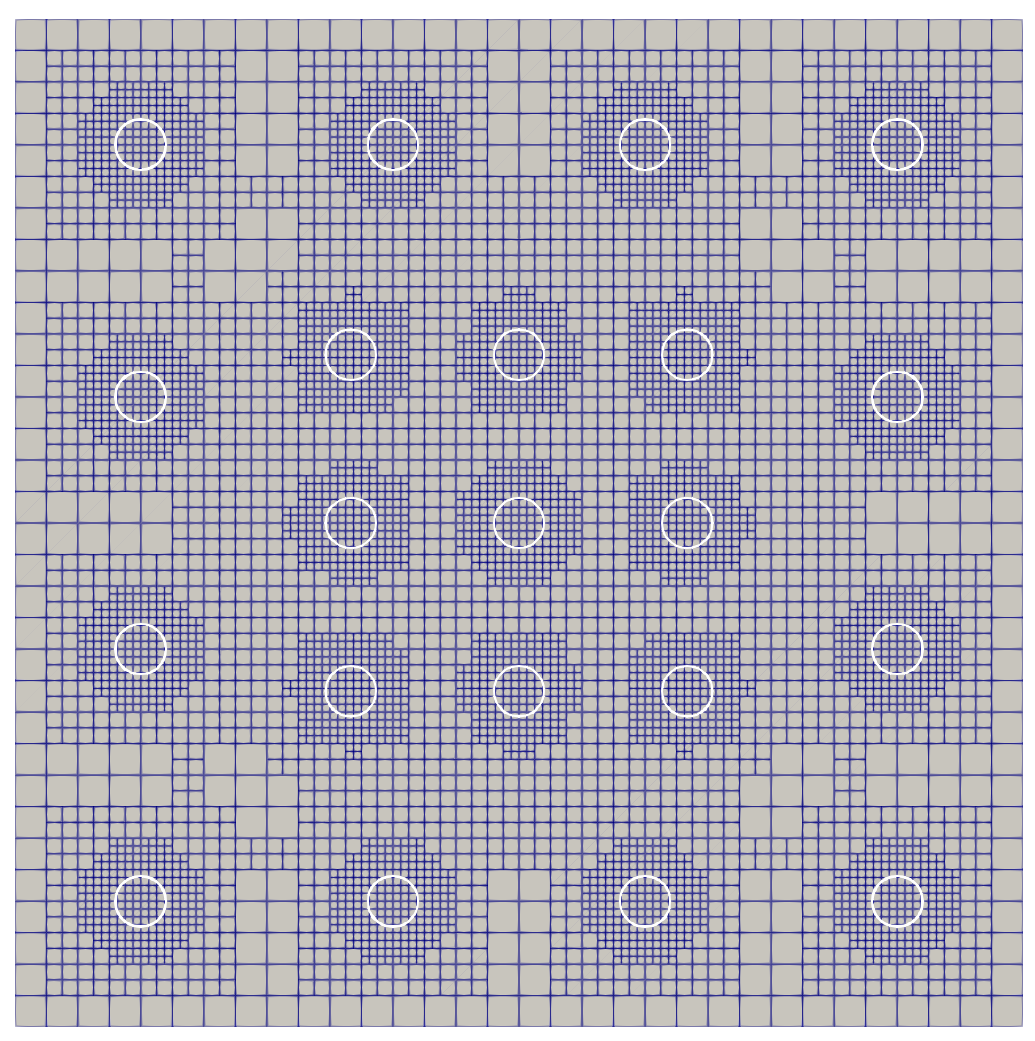}
	\caption{Refined mesh for the three-dimensional extrusion of case (a) from Figure \ref{fig:test4_core}. The number of degrees of freedom is around $2 \times 10^6$ for the underlying solid mesh and $100 m$ for the inclusions ($m = 21$ in the case shown here).}\label{fig:test3D_setup}
\end{figure}
Considering the third dimension allows us to analyze inclusion-induced stresses in more details.
Figure \ref{fig:test3D_core} shows the Normal and tangential stresses for the face 0, which is 
parallel to the $(y,z)$-plane, and to the face 4, parallel to the $(x,y)$-plane, and hence perpendicular to the inclusions (see also Figure \ref{fig:test45_setup}).
The results are in line with the
outcome of the two-dimensional example, with an increase of
the normal stresses when increasing the volume occupied by the fluid and the imposed displacement. 

\begin{figure}[!h]
	\centering
	\includegraphics[height=.3\textwidth, page=1]{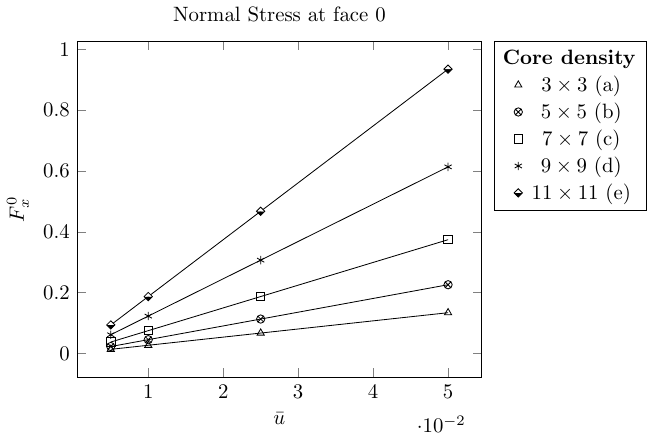}
	\hspace{0.2cm}
	\includegraphics[height=.3\textwidth, page=3]{standalonefigures/test3D_parallel.pdf}
	\\
	\includegraphics[height=.3\textwidth, page=4]{standalonefigures/test3D_parallel.pdf}
	\hspace{0.2cm}
	\includegraphics[height=.3\textwidth, page=5]{standalonefigures/test3D_parallel.pdf}
	\caption{Normal and Tangential stresses on face 0 (normal to $x$ direction) and face 4 (normal to $z$ direction).}\label{fig:test3D_core}
\end{figure}

The final example considers, for the three-dimensional domain different computational experiments varying also the lateral compression imposed on the tissue sample. 
Namely, for different vessel densities, we apply a given compression on the faces 0, 1, 2, and 3 (see Figure \ref{fig:test45_setup}), while
prescribing homogeneous Neumann boundary conditions on faces 4 and 5
(perpendicular to the inclusions).
An example of the obtained numerical results is shown in Figure \ref{fig:test3D_test}.
\begin{figure}[!ht]
	\centering
	\includegraphics[height=.3\textwidth]{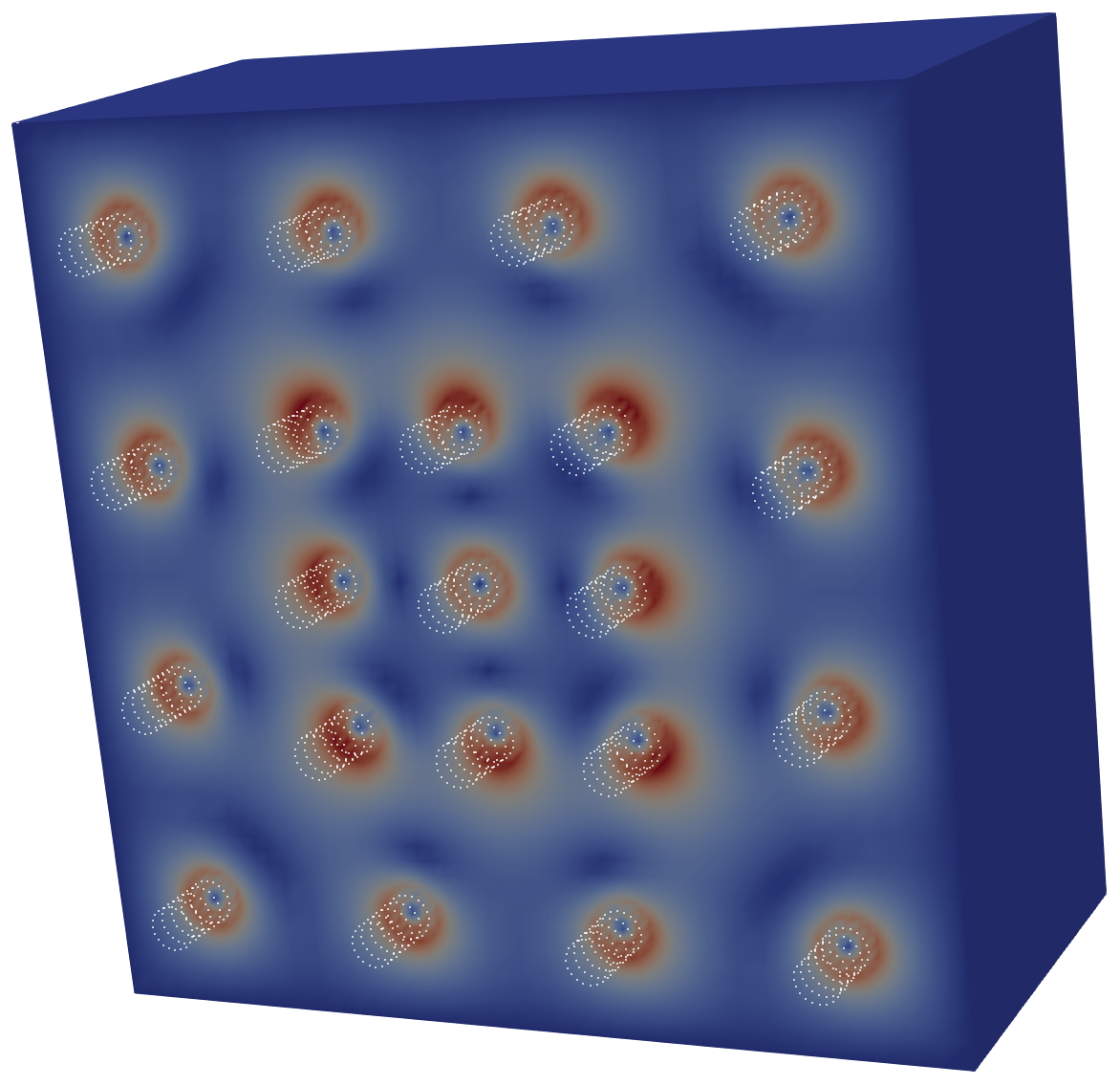}
	\includegraphics[height=.3\textwidth]{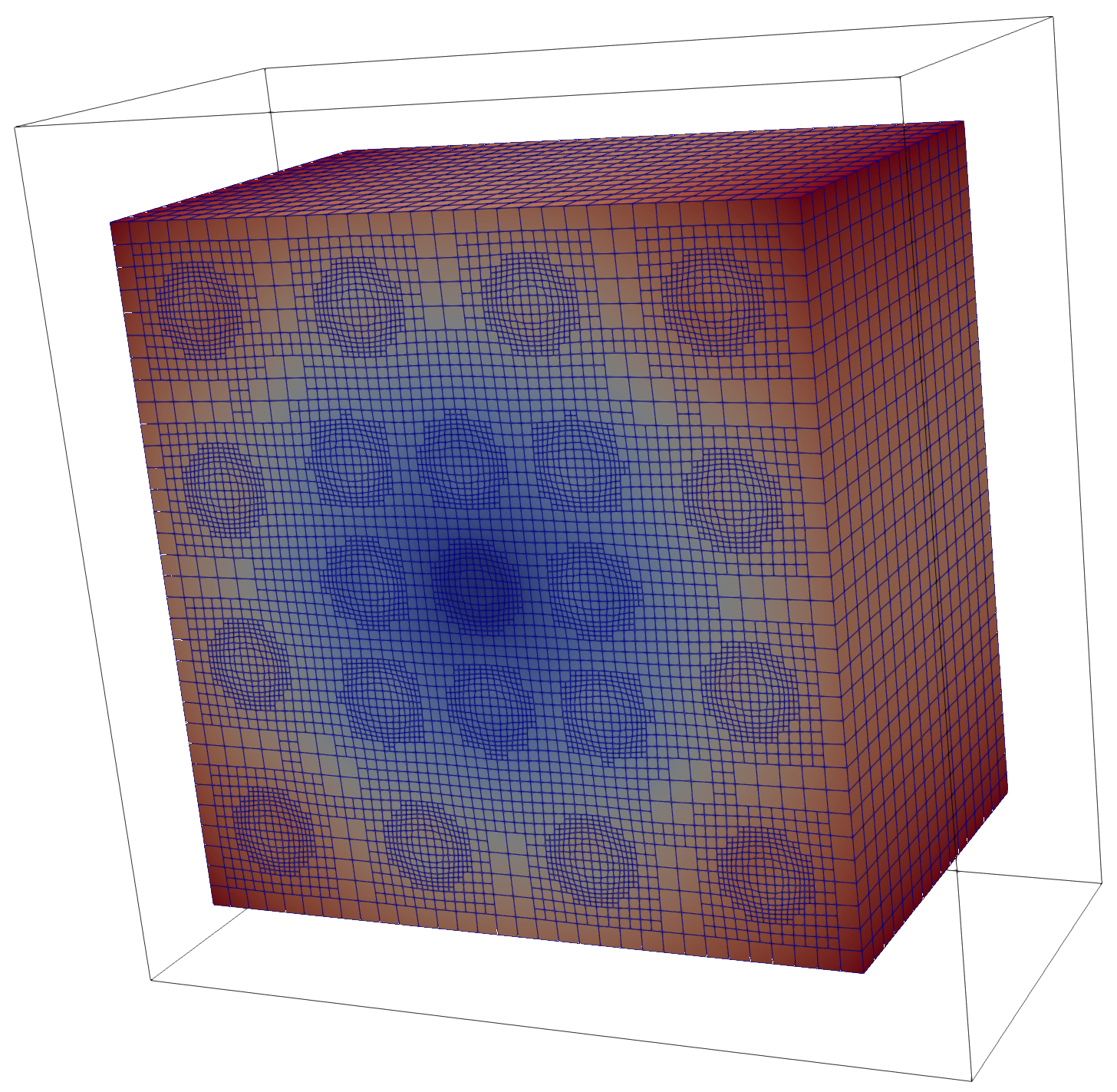}
	\caption{Three-dimensional model with vessel expansion and lateral compression. 
    Left: Results for $\bar{u} = 0.025$ (vessel expansion) and no compression.
    Right: Results for $\bar{u} = 0.025$ (vessel expansion) and lateral compression $c=0.15$.}\label{fig:test3D_test}
\end{figure}

The results are
summarized in Figure \ref{fig:test3D_core_pressure}. 
In this experiment, the slope of the normal inclusion-induced stress on the lateral faces can be related to the compressibility modulus of the effective material (solid matrix + vessels). The simulation shows that
the inclusions add a non-negligible contribution to the compressibility,
hence confirming that handling the different phases consistently can be critical when modeling and estimating mechanical behaviors.
\begin{figure}[!ht]
	\centering
	\includegraphics[height=.3\textwidth, page=11]{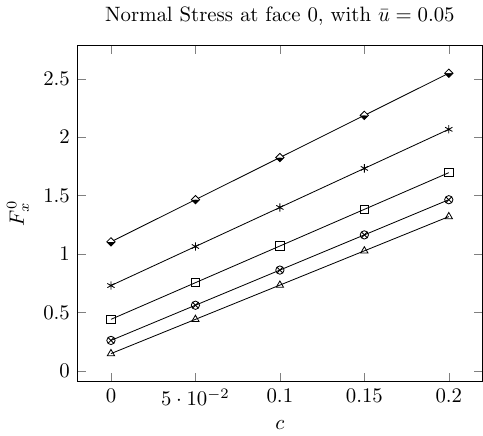}
	\hspace{0.05cm}
	\includegraphics[height=.3\textwidth, page=15]{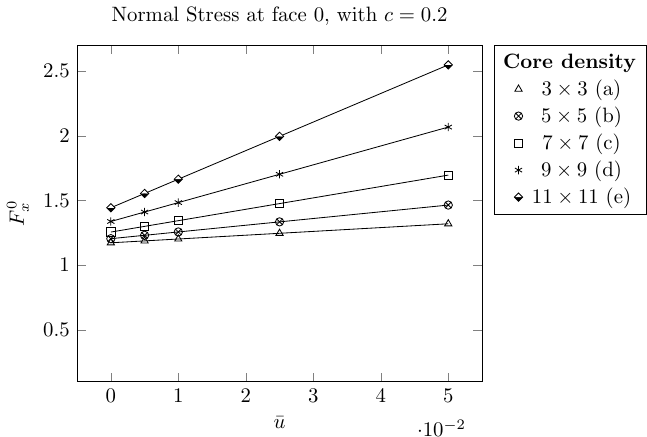}
	\caption{Normal inclusion-induced stress on face 0 (perpendicular to $x$-axis) for different configurations, varying the displacement imposed on the vessel interface ($\bar{u}$), and the magnitude of the compression imposed on the lateral faces ($c$).}
	\label{fig:test3D_core_pressure}
\end{figure}

\section{Conclusions} \label{sec:conclusions}
We have proposed, analyzed, and numerically tested a computational 
approach for the simulation of multiscale problems involving the interplay of linear elastic solid and slender fluid inclusions.  
The method handles the inclusions as low-dimensional immersed boundaries within the tissue mesh, combined with the recently proposed reduced Lagrange multiplier approach \cite{LHPZ}, which allows to efficiently handle the mixed-dimensionality of the problem.
From the practical perspective, the usage of the immersed methods allows to reduce the overall complexity of the problem since the explicit discretization of the inclusion interface within the finite element mesh is not required, while the combination with the model-order reduction
allows to efficiently handle the physics of the considered problem.

We extended the method of \cite{LHPZ}, originally proposed for scalar problems, to the setting of three-dimensional linear elasticity coupled to one-dimensional blood flow. 
This problem requires a non-standard conditions at the fluid-solid interface to consistently handle the fact that one-dimensional displacement field imposed only along the normal direction, and that it is invariant to the the macroscale deformation. We showed that this condition can be naturally imposed within the reduced Lagrange multiplier framework by properly selecting the Lagrange multipliers space that guarantee well-posedness of the
continuous formulation.

One limitation of the model considered in this is that it does not account for the the back coupling of the solid matrix onto the fluid vasculature. The extension of the model to this more general setting, based on the preliminary results of Heltai \textit{et al.}\cite{heltai-caiazzo-mueller-2021}, is currently subject of ongoing research.

We assessed the performance of the proposed scheme by validating the expected convergence rates in a two-dimensional model problem with a single inclusion. As next, we considered different cases in two and three dimensions with multiple inclusions.  Our tests indicated as well that as the scale separation increases (thinner inclusions), a reduced-order space of dimension $N=2$ is sufficient for a valid approximation. 
Additionally, we performed a detailed study of the influence of microscale quantities (inclusion distribution) on the  effective mechanical parameters. The results%
demonstrate  that  tissue response is sensitive to variations of vascular parameters and
fluid dynamics, potentially inducing non-linear mechanical responses in presence of inhomogeneities. 
These results show that the proposed multiscale model can be effectively used for the numerical investigation and for the numerical upscaling of multiscale materials, and 
align with recent discoveries that have emphasized the intricate interconnections between effective parameters at microscale 
\cite{safraou-etal-2023,JAITNER2025312}. The application of the model in the context of inverse problems in medical tissue imaging will be one of the topics to be considered in upcoming works.

\section*{Author contributions}
C. Belponer: Formal analysis, Software, Validation, Writing, Visualization \\
A. Caiazzo: Conceptualization, Formal analysis, Methodology, Writing \\
L. Heltai:  Conceptualization, Formal analysis, Methodology, Software, Writing

\section*{Financial disclosure}
The research of C. Belponer has been funded by the Deutsche
Forschungsgemeinschaft (DFG, German Research Foundation), grant DFG CA 1159/1-4
and PE 2143/1-6. L. Heltai acknowledges the partial support of the grant MUR
PRIN 2022 No. 2022WKWZA8 "Immersed methods for multiscale and multiphysics
problems (IMMEDIATE)". The authors acknowledge partial support of the European
Research Council (ERC) under the European Union's Horizon 2020 research and
innovation programme (call HORIZON-EUROHPC-JU-2023-COE-03, grant agreement No.
101172493 ``dealii-X: an Exascale Framework for Digital Twins of the Human
Body''). LH acknowledges the MIUR Excellence Department Project awarded to the
Department of Mathematics, University of Pisa, CUP I57G22000700001. LH is member
of Gruppo Nazionale per il Calcolo Scientifico (GNCS) of Istituto Nazionale di
Alta Matematica (INdAM).

\section*{Conflict of interest}

The authors declare no potential conflict of interests.

\bibliographystyle{abbrv}
\bibliography{bibliography.bib}

\section*{Supporting information}
The source code used in this research, along with the simulation input data, is openly available on the GitHub repository \url{https://github.com/luca-heltai/reduced_lagrange_multipliers}. This repository contains both the implementation of the algorithms and methods described in the paper, as well as the simulation input data required to reproduce the experiments.

\end{document}